\documentclass[a4paper]{amsart}
\usepackage{a4wide}
\usepackage[dvips]{graphicx}
\usepackage{color}
\usepackage{amssymb}
\usepackage{textcomp}

\let\oldtocsection=\tocsection

\let\oldtocsubsection=\tocsubsection

\let\oldtocsubsubsection=\tocsubsubsection

\renewcommand{\tocsection}[2]{\hspace{0em}\oldtocsection{#1}{#2}}
\renewcommand{\tocsubsection}[2]{\hspace{1em}\oldtocsubsection{#1}{#2}}
\renewcommand{\tocsubsubsection}[2]{\hspace{2em}\oldtocsubsubsection{#1}{#2}}

\usepackage{amsmath}
\usepackage{amsfonts}
\usepackage{latexsym}
\usepackage{amscd}
\usepackage{graphicx}
\usepackage[latin1]{inputenc}
\usepackage[english]{babel}
\usepackage{enumerate}
\usepackage{afterpage}
\newcounter{notes}

\def\a{\alpha}
\def\b{\beta}
\def\d{\delta}

\def\g{\gamma}
\def\l{\lambda}

\newcommand{\T}{\mathcal{T}}
\newcommand{\N}{\mathbb{N}}

\renewcommand{\to}{\longrightarrow}
\def\co{\colon\thinspace}

\newcommand{\Uom}{\boldsymbol{\omega}}
\newcommand{\Tc}{\mathcal{T}}

\newcommand{\C}{\mathbb{C}}
\newcommand{\Z}{\mathbb{Z}}

\newcommand{\R}{\mathbb{R}}
\newcommand{\PSL}{\mathrm{PSL}}
\newcommand{\Out}{\mathrm{Out}}

\newcommand{\SLtwoC}{\mathrm{SL}(2,\C)}
\newcommand{\SLtwoR}{\mathrm{SL}(2,\R)}

\newcommand{\MCG}{\mathcal{MCG}}

\newcommand{\X}{\mathfrak{X}}

\newcommand{\Sc}{\mathcal{S}}
\DeclareMathOperator{\Tr}{\mathrm{tr}}

\newtheorem{Theorem}{Theorem}[section]
\newtheorem{Lemma}[Theorem]{Lemma}
\newtheorem{Proposition}[Theorem]{Proposition}
\newtheorem{Corollary}[Theorem]{Corollary}
\newtheorem{introthm}{Theorem}

\newtheorem{Definition}[Theorem]{Definition}
\newtheorem{Remark}[Theorem]{Remark}

\begin{document}

\title{On the character variety of the three--holed projective plane}

\author{Sara Maloni}
\address{Department of Mathematics, Brown University}
\email{sara\_maloni@brown.edu}
\urladdr{http://www.math.brown.edu/$\sim$maloni}

\author{Fr\'{e}d\'{e}ric Palesi}
\address{Aix Marseille Universit\'{e}, CNRS, Centrale Marseille, I2M, UMR 7373, 13453 Marseille, France}
\email{frederic.palesi@univ-amu.fr}
\urladdr{www.latp.univ-mrs.fr/$\sim$fpalesi}

\begin{abstract}
We study the (relative) $\SLtwoC$ character varieties of the three-holed projective plane and the action of the mapping class group on them. We describe a domain of discontinuity for this action, which strictly contains the set of primitive stable representations defined by Minsky, and also the set of convex-cocompact characters. We consider the relationship with the previous work of the authors and S. P. Tan on the character variety of the four-holed sphere. 
\end{abstract}

\maketitle

\tableofcontents

\section{Introduction}\label{s:intro}

In this article we continue the study of the character variety $\mathfrak{X} = \mathfrak{X}(F_3, \SLtwoC)$ started in \cite{mal_ont}, in joint work with Ser Peow Tan. 

Character varieties $\X(\Gamma, G)$, which are the (geometric invariant) quotient $\mathrm{Hom}(\Gamma, G)//G$ of the spaces of representations of a word hyperbolic group $\Gamma$ into a semi-simple Lie group $G$ by conjugation, have been extensively studied. Here we will focus on the study of the action of the outer automorphism group $\Out(\Gamma)$ on $\X(\Gamma, G)$ given by $\theta([\rho]) = [\rho \circ \theta^{-1}]$. This question is motivated by the classical example of the proper discontinuous action of the mapping class group $\MCG(\Sigma)$ on the Teichm\"uller space $\T(\Sigma)$ of a closed orientable surface $\Sigma$. In fact, $\T(\Sigma)$ corresponds to the connected component of  $\mathfrak{X}(\pi_1(\Sigma), \PSL_2(\R))$ consisting of discrete and faithful representations of $\pi_1(\Sigma)$ into $\PSL_2(\R)$, and $\MCG(\Sigma)$ is an index $2$ subgroup of the outer automorphism group $\Out(\pi_1(\Sigma))$. It is also conjectured that $\T(\Sigma)$ is the biggest domain of discontinuity for the $\MCG(\Sigma)$-action. See Canary \cite{can_dyn} for a very interesting survey on this topic.

If one considers surfaces $\Sigma_{g,b}$ with non empty boundary, then the fundamental group $\pi_1(\Sigma_{g,b})$ is a free group $F_{n}$ and the mapping class group $\MCG(\Sigma_{g,b})$ is a subgroup of $\Out(F_n)$. While the action of $\Out(F_n)$ on $\X$ is well-known to be properly discontinuous on the set of discrete, faithful, convex-cocompact (i.e. Schottky) characters, the action on the complement of these characters is more mysterious. Minsky \cite{min_ond} studied this action, and described the set of {\em primitive-stable representations} $\X_{\mathrm{ps}}$--the ones such that the axes of primitive elements are uniform quasi-geodesics. He proved that $\X_{\mathrm{ps}}$ is an open domain of discontinuity for the action which is strictly larger than the set of discrete, faithful, convex-cocompact (i.e. Schottky) characters.

Another approach in the study of the character varieties $\mathfrak{X}(F_n, \SLtwoC)$ was introduced by Bowditch in \cite{bow_mar} and this approach was later generalized by Tan, Wong and Zhang \cite{tan_gen}, and by the authors and Tan \cite{mal_ont}, among others. They defined a domain of discontinuity $\X_{Q}$, the Bowditch set of representations, which contains the set $\X_{\mathrm{ps}}$, and hence is also strictly larger than the set of discrete, faithful, convex-cocompact (i.e. Schottky) characters. Bowditch's idea was to use a combinatorial viewpoint using trace functions on simple closed curves. In \cite{mal_ont} the authors and Ser Peow Tan generalised those methods and studied the case $n = 3$. In \cite{mal_ont} we viewed the free group of rank three $F_3$ as the fundamental group of the four-holed sphere, while in this article we will consider $F_3$ as the fundamental group of the three-holed projective plane.

\vskip 5pt

By a classical result on character varieties, see for example Fricke and Klein \cite{fri_vor}, the variety $\mathfrak{X}$ can be identified with the set of septuples $(a,b,c,d, x,y,z) \in \C^7$ such that:
\begin{equation} \label{vertex}
  a^2+b^2+c^2+d^2+abcd = x(ab+cd)+y(bc+ad)+z(ac+bd)+4-x^2-y^2-z^2-xyz,
\end{equation}
where  $a= \Tr( \rho (\a)) $, $b = \Tr (\rho (\b))$, $c = \Tr (\rho (\g ))$, $d = \Tr (\rho (\a\b\g))$, $x= \Tr( \rho (\a\b))$, $y = \Tr (\rho (\b\g))$, $z = \Tr (\rho (\a\g))$. 

\begin{figure}
[hbt] \centering
\includegraphics[height=4 cm]{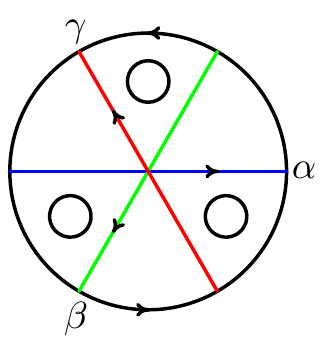}
\includegraphics[height=4 cm]{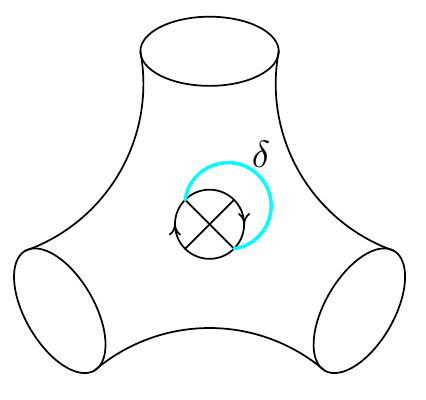}
\caption{The space $N_{1, 3}$ (in two different pictures) and the simple closed curves corresponding to $\a$ (in blue), $\b$ (in green), $\g$ (in red), $\d$ (in cyan).}
\label{fig:n13}
\end{figure}

We can identify $F_3=\langle \alpha,\beta,\gamma,\delta \mid \alpha \beta \gamma \delta \rangle$ with the fundamental group of a three-holed projective plane $N$ so that $\alpha\beta, \beta\gamma, \alpha\gamma$ correspond to the three boundary components $\partial N$ of $N$, see Figure \ref{fig:n13}. In this case, the relative character variety $\mathfrak{X}_{(x,y,z)}(N)$, that is, the set of (classes of) representations for which the traces of the boundary curves are fixed, can be represented as quadruples $(a,b,c,d)\in \C^4$ which satisfy \eqref{vertex}: $$\mathfrak{X}_{(x,y,z)}(N) = \{ (a,b,c,d)\in \C^4 \mid \mbox{equation }\eqref{vertex} \mbox{ holds.} \}.$$
On the other hand, if we identify $F_3$ to the fundamental group of a four-holed sphere $S$, with the elements $\alpha, \beta, \gamma, \delta$ identified with $\partial S$, as we did in \cite{mal_ont}, then the relative character variety $\mathfrak{X}_{(a,b,c,d)}(S)$ is identified with triples $(x,y,z)\in \C^3$ satisfying \eqref{vertex}:
$$\mathfrak{X}_{(a,b,c,d)}(S) = \{ (x,y,z)\in \C^3 \mid \mbox{equation }\eqref{vertex} \mbox{ holds.} \}.$$ This shows that these two points of view are somehow `dual' to one another.

When considering $F_3$ as the fundamental group of the three-holed projective plane $N$ or of the four-holed sphere $S$, one can study the dynamics of the action of the (pure) mapping class groups $\mathcal{MCG}(N)$ and $\mathcal{MCG}(S)$, which are proper subgroups of $\Out(F_3)$. In \cite{mal_ont}, the authors and Ser Peow Tan studied the action of $\mathcal{MCG}(S)$ on $\X$ and described a domain of discontinuity $\X_Q(S)$ which strictly contains $\X_{\mathrm{ps}}$. More precisely, we studied the action of the group $\mathbb{Z}_2 \star \mathbb{Z}_2 \star \mathbb{Z}_2$ on $\mathfrak{X} = \{ (a,b,c,d, x, y, z)\in \C^7 \mid \mbox{equation }\eqref{vertex} \mbox{ holds.}\}$ generated by the following involutions:
\begin{equation}\label{eqn:theta_x}
\begin{aligned}
  \theta_x(a,b,c,d,x,y,z)&=&(a,b,c,d, p-yz-x, y,z),\\
   \theta_y(a,b,c,d,x,y,z)&=&(a,b,c,d, x, q-xz-y,z),\\
   \theta_z(a,b,c,d,x,y,z)&=&(a,b,c,d, x, y,r-xy-z),
\end{aligned}
\end{equation}

where $$p=ab+cd, \quad q=bc+ad, \quad r=ac+bd, \quad s=4-a^2-b^2-c^2-d^2-abcd.$$
The mapping class group $\MCG(S)$ is an order $2$ subgroup of the group $\mathbb{Z}_2 \star \mathbb{Z}_2 \star \mathbb{Z}_2$ defined above. These involutions are defined by exchanging the two solutions of Equation $\eqref{vertex}$, considered as a quadratic equation in one of the variable $x$, $y$ or $z$, respectively.

In this article, we focus on the dynamics of $\mathcal{MCG}(N)$ on $\X$. As for the four-holed sphere case, we actually study the action of $\mathbb{Z}_2 \star \mathbb{Z}_2 \star \mathbb{Z}_2 \star \mathbb{Z}_2$ given by the following involutions:
\begin{equation}\label{eqn:theta_a}
\begin{aligned}
  \theta_a(a,b,c,d,x,y,z)&=&(xb+zc+yd-bcd-a,b,c,d, x, y,z),\\
   \theta_b(a,b,c,d,x,y,z)&=&(a,xa+yc+zd-acd-b,c,d, x, y,z),\\
   \theta_c(a,b,c,d,x,y,z)&=&(a,b,za+yb+xd-abd-c,d, x, y,z),\\
   \theta_d(a,b,c,d,x,y,z)&=&(a,b,c,ya+zb+xc-abc-d, x, y,z).
\end{aligned}
\end{equation}

The product of two of these involution correspond to a Dehn twist about a $2$--sided simple closed curve in $N$ (as it was the case for the four-holed-sphere). This new point of view is very interesting. Among other reasons, it turns out that the Torelli subgroup $\Tc_n$ of $\mathrm{Out}(F_3)$ is an index two subgroup of the group generated by the seven involutions $\theta_a, \theta_b, \theta_c, \theta_d, \theta_x, \theta_y, \theta_z$, as we will prove in details in Appendix \ref{app:Torelli}. See also Remark \ref{Torelli}. We hope to combine these two approaches in a future paper, and study the action of the Torelli subgroup on the full character variety $\mathfrak{X}$.

The main result of this paper is the following:

\begin{introthm}\label{Main}
  There exists an open domain of discontinuity $\X_Q(N)$ for the action of $\MCG(N)$ on $\X$ which strictly contains $\X_{\mathrm{ps}}$ (and so the set of discrete, faithful, convex-cocompact characters).
\end{introthm}

The set $\X_Q(N)$ can be described as the set of representations $\rho \in \X$ satisfying the following conditions:
\begin{enumerate}[{(BQ1)}]
  \item $\forall \g \in \Sc_2$, $\Tr\rho(\g) \not\in [-2,2]$ ; and
  \item  $ \forall K>0$, $\; \#\{\g\in\Sc_2 \mid |\Tr\rho(\g)| \le K\} < \infty , $
\end{enumerate}
where $\Sc_2$ is the set of free homotopy classes of unoriented $2$--sided simple closed curve in $N$.

In Section \ref{s:BQ} we will use another equivalent definition for $\X_Q(N)$ which is more complicated to state, but is necessary to prove the main theorem. (The proof of the equivalence between the two definitions is contained in Section \ref{s:characterization}). This equivalent definition is also useful if one wants to write a computer program which draws slices of the domain of discontinuity, because it uses a unique constant $K$ for condition $(BQ2)$. In addition, in the same section we will also prove that the set $\X_Q(N)$ can be equivalently defined in terms of the growth of the elements of $F_3$ corresponding to ($2$--sided) simple closed curves in $N$. In oder to study these asymptotic growth, we use the {\em Fibonacci} function, a `reference' function which can be defined recursively starting from the values around an edge, and which happens to be related to the word length of the elements of $\Gamma$ representing these simple closed curves, see Proposition \ref{prop:word}. This will have as a corollary, the fact that $\X_Q(N)$ contains the set $\X_{\mathrm{ps}}$ of primitive-stable representations, see Proposition \ref{pr:primitivestable}.

The strategy to prove Theorem \ref{Main} consists mainly in a careful analysis using trace functions on simple closed curves in $N$, as previously done in Bowditch \cite{bow_mar}, Tan--Wong--Zhang \cite{tan_gen} and Maloni--Palesi--Tan \cite{mal_ont}. The main idea is to define a combinatorial graph $\Upsilon$, the dual of the complex of curves of $N$, and define for any representation $\rho\in \X$, an orientation on the $1$--skeleton of $\Upsilon$, and prove the existence of an attracting subtree. Then, we prove that this attracting subtree is finite if and only if $\rho \in \X_Q(N)$, and we use that to show that $\X_Q(N)$ is open and the action of $\MCG(N)$ on it is properly discontinuous.

\vskip 5pt

After setting up the notation and the required background that we need in Section \ref{s:not},  we describe the combinatorial view point we adopt in Section \ref{s:analysis}, which uses trace functions on simple closed curves, and we describe the construction of the attracting subtree. We conclude in Section \ref{s:fibonacci} after understanding the asymptotic growth of the length of simple closed curves. In the same Section we also give different characterizations of the set $\X_Q(N)$. In Appendix \ref{app:functionH} we will prove an explicit formula related to the asymptotic growth of the representations along two sided simple closed curves, while in Appendix \ref{app:Torelli} we will prove that the Torelli subgroup $\Tc_n$ of $\mathrm{Out}(F_3)$ is an index two subgroup of the group generated by the seven involutions $\theta_a, \theta_b, \theta_c, \theta_d, \theta_x, \theta_y, \theta_z$.

\vskip 5pt

\noindent {\bf Acknowledgements.} The authors acknowledge support from U.S. National Science Foundation grants DMS 1107452, 1107263, 1107367 RNMS: ``Geometric Structures and Representation Varieties'' (the GEAR Network). The second author was partially supported by the European Research Council under the European Community's seventh Framework Programme (FP7/2007-2013)/ERC grant agreement n\textdegree\;\; FP7-246918, and by ANR VALET (ANR-13-JS01-0010) and the work has been carried out in the framework of the Labex Archimede (ANR-11-LABX-0033) and of the A*MIDEX project (ANR-11-IDEX-0001-02).

This material is based upon work supported by the National Science Foundation under Grant No. 0932078 000 while the authors were in residence at the Mathematical Sciences Research Institute in Berkeley, California, during the Spring 2015 semester. The authors are grateful to the organizers of the program for the invitations to participate, and to the MSRI and its staff for their hospitality and generous support.

\section{Notations}\label{s:not}

In this section we fix the notations which we will use in the rest of the paper and give some important definitions. Since in this article we will study representation $\rho\co F_3\to \SLtwoC$ from the free group on three generators $F_3$ into $\SLtwoC$, in the same way as in the article of Maloni, Palesi, Tan \cite{mal_ont}, we will follow the notation and structure of that paper. When possible, we will try to simplify the arguments, stating more clearly the relations between the different results and the idea behind the proofs. Note that this work, as well as our previous work \cite{mal_ont}, are influenced by Bowditch's results \cite{bow_mar}, which were generalized by Tan, Wong and Zhang \cite{tan_gen}. Note also that Huang--Norbury \cite{hua_sim} studied the particular case of the three-holed projective plane where all the boundary components are punctures, which make equation \eqref{vertex} symmetric and much simpler. Their results go in a different direction than ours: we are interested in dynamical questions, while Huang and Norbury are more interested in some topological questions, as the study of a McShane's identities, or of systoles of $N$.

\subsection{The three-holed projective plane $N$ and its fundamental group $\Gamma$}

Any (compact) non-orientable surface $N_{g,b}$ is characterized (topologically) by the number $g$ of cross-caps and the number $b$ of boundary components.

Let $N = N_{1, 3}$ be a (topological) three-holed (real) projective plane, namely, a projective plane with three disjoint open disks removed, and let $\Gamma$ be its fundamental group. The group $\Gamma$ is isomorphic to the free group on three generators $\Z \ast \Z \ast \Z$ and admits the following presentation
$$\Gamma = \langle \a, \b, \g, \d \mid \a\b\g\d \rangle,$$
where $\a$, $\b$, $\g$, and $\d$ are the loops described in Figure \ref{fig:n13}.

Note that with these generators, the homotopy classes of the three boundary components, one for each removed disk, correspond to the elements $\a\b$, $\b\g$, and $\a\g$.

We define an equivalence relation $\sim$ on $\Gamma$ by: $g \sim h$ if and only if $g$ is conjugate to $h$ or $h^{-1}$. Then $\Gamma/\!\sim$ can be identified with the set of free homotopy classes of unoriented closed curves on $N$.

\subsection{Simple closed curves on $N$}\label{ss:simple}

Let $\Sc = \Sc(N)$ be the set of free homotopy classes of {\em essential} simple closed curves on $N$. Recall that a curve is \emph{essential} if it does not bound a disc, an annulus or a M\"obius strip. We will omit the word essential from now on. We can then identify $\Sc$ to a well-defined subset of $\Gamma/\!\sim$.  

Simple closed curves in a non-orientable surface are of two types: a simple closed curve is said to be $1$--{\em sided} if its tubular neighborhood is homeomorphic to a M\"obius strip, and $2$--{\em sided} if the neighborhood is homeomorphic to an annulus. The $4$ curves $\a$, $\b$, $\g$, and $\d$ described above are all $1$--sided. Let $\Sc_i$, where $i = 1, 2$, be the subset of $\Sc$ corresponding to $i$--sided simple closed curves. We recall that Dehn twists can only be defined along $2$--sided simple closed curves.

\begin{Remark}\label{2sided}
  Since it will be important later, we notice that in $N = N_{1,3}$, there is a $1$--to--$1$ correspondence between:
  \begin{itemize}
  	\item (unordered) pairs $(\a, \b)$ of (free homotopy classes of) $1$--sided simple closed curves intersecting exactly once; and 
	\item (free homotopy classes of) $2$--sided simple closed curves $\xi_{\a, \b}$.
\end{itemize} 
We will say that $\xi_{\a, \b}$ is {\em associated} with the pair $(\a,\b)$.
\end{Remark}
\begin{proof}
  In fact, the $R$--neighborhood of any pair $(\a, \b)$ of $1$--sided simple closed curves intersecting once corresponds to an embedded two holed projective plane M. One of the boundary of $M$ is homotopic to a boundary component of $N$, and we denote by $\xi_{\a, \b}$ the other boundary curve, which is an essential $2$--sided curve in $N$, and which corresponds to the element $\a\b^{-1}$. Reciprocally, any $2$--{\em sided} essential simple closed curve on $N$ splits the surface into a pair of pants and a two-holed projective plane $M$, and there are exactly two $1$--{\em sided} curves in $M$. 
\end{proof}
  
\subsection{Relative character variety $\mathfrak{X}_{\Uom}(N)$}

The character variety $\mathfrak{X} = \mathfrak{X}(\Gamma, \SLtwoC)$ is the space of equivalence classes of representations $\rho \co \Gamma \rightarrow \SLtwoC$, where the equivalence classes are obtained by taking the closure of the orbit under the conjugation action by $\SLtwoC$. As mentioned in the Introduction, a classical result on the character varieties (see, for example, (9) in p. 298 of Fricke and Klein \cite{fri_vor}) states that the map

\begin{align*}
	f : \mathfrak{X} & \longrightarrow \, \, \, \C^7 \\
		[ \rho ] & \longmapsto 
		\begin{pmatrix} a \\ b \\ c \\ d \\ x \\ y \\ z \end{pmatrix} = 
		\begin{pmatrix} \Tr( \rho (\a)) \\ \Tr( \rho (\b))\\ \Tr( \rho (\g))\\ \Tr( \rho (\d))\\ \Tr( \rho (\a\b))\\ \Tr( \rho (\b\g))\\ \Tr( \rho (\a\g)) 				\end{pmatrix}
\end{align*}
 provides an identification of  $\mathfrak{X}$ with the set
    \begin{equation} 
        \left\{(a,b,c,d, x,y,z) \in \C^7 \mid \mbox{equation } (1) \mbox{ holds} \right\}.
    \end{equation}

Let $\Uom = (x,y,z) \in \C^3$. A representation $\rho\co \Gamma \to \SLtwoC$ is said to be a $\Uom$--{\it representation}, or $\Uom$--{\it character}, if, for some fixed generators $\a, \b, \g \in \Gamma$, 
we have $$\begin{aligned}
   \Tr \rho(\a\b) = x,\\ 
   \Tr \rho(\b\g) = y,\\ 
   \Tr \rho(\a\g) = z.
\end{aligned}$$ 
The space of equivalence classes of $\Uom$-representations is denoted by $\mathfrak{X}_{\Uom}$ and is called the $\Uom$--\emph{relative character variety}. These representations correspond to representations of the three-holed projective plane where we fix the conjugacy classes of the three boundary components in the space of closed orbits. The previous map gives an identification  of $\mathfrak{X}_{\Uom}$ with the set 
$$\left\{(a,b,c,d) \in \C^4 \mid  \mbox{equation } (1) \mbox{ holds}  \right\}. $$
    
\subsection{The mapping class group $\MCG(N)$}
   
The pure mapping class group $\MCG = \MCG(N) :=\pi_0({\rm Homeo}(N))$ of $N$ is the subset of the group of isotopy-classes of homeomorphisms of $N$ fixing the boundary components pointwise. Huang and Norbury gave a complete description of this group in \cite[Section 2.7]{hua_sim} and in particular, they show that $\MCG \cong F \rtimes \mathrm{Stab}(\Delta)$, where $F$ is the group generated by the four involutions $\theta_a, \ldots, \theta_d$ defined in the Introduction, and $\mathrm{Stab}(\Delta)\cong \mathbb{Z}_2 \times \mathbb{Z}_2$. Since $F$ is a finite index subgroup of $\MCG (N)$, we will study its action on the character varieties $\X$ and $\X_{\Uom}$, as it is much simpler to describe.

\begin{Remark}\label{Torelli}
	The group generated by the seven involutions $\theta_a, \theta_b, \theta_c, \theta_d, \theta_x, \theta_y, \theta_z$ has a remarkable interpretation: it is a $\Z_2$ extension of the Torelli group $\Tc_3$ of the free group of rank three. Since this fact is of independent interest, we prove it in Appendix \ref{app:Torelli}.
\end{Remark}

\subsection{The binary tree $\Upsilon$}

In \cite{mal_ont} we considered the countably infinite simplicial tree $\Sigma$ properly embedded in the plane all of whose vertices have degree 3 in order to define $\mu$--Markoff triples and $\mu$--Markoff maps. The graph $\Sigma$ is the simplicial dual to the Farey graph, which coincides with the complex of curves for the four-holed sphere $S$. We define here the analog for the three-holed projective plane case.

Let $\mathcal{CC}(N)$ be the complex of curves of $N$, which is the $3$--dimensional abstract simplicial complex, where the $k$--simplices are given by subsets of $k+1$ distinct (homotopy classes of) $1$--sided simple closed curves in $N$ that pairwise intersect once. See Scharlemann \cite{sch_the} for a more detailed discussion on the complex of curves of non-orientable surfaces.

Let $\Upsilon$ be the simplicial dual to $\mathcal{CC}(N)$, as described by Bowditch \cite{bow_mar}. In particular $\Upsilon$ is  a countably infinite simplicial tree properly embedded in the hyperbolic $3$--space all of whose vertices have degree $4$. Let $\Upsilon^{(k)}$ denote the set of $k$--simplices in $\Upsilon$. Let $i(\cdot, \cdot)$ denote the geometric intersection number between two curves, that is the minimal number of intersections in the homotopy classes of the curves. Note that the minimal intersection number between two $1$--sided curves in $N$, is $1$. We have the following sets:
\begin{itemize}
  \item $\Upsilon^{(0)} = \{(\a_1, \a_2, \a_3, \a_4) \mid \a_i \in \Sc_1, i(\a_i,\a_j) = 1 \text{ if } i \neq j\}$;
  \item $\Upsilon^{(1)} = \{(\a_1, \a_2, \a_3) \mid \a_i \in \Sc_1, i(\a_i,\a_j) = 1 \text{ if } i \neq j\}$;
  \item $\Upsilon^{(2)} = \{(\a_1, \a_2) \mid \a_i \in \Sc_1, i(\a_1,\a_2) = 1\}$ which is also  $\{\xi\in\Sc_2\}$;
  \item $\Upsilon^{(3)} = \{\a \in \Sc_1\}$.
\end{itemize}

So each vertex [resp. edge, face, or region] correspond to a quadruple [resp. triple, pair, or singleton] of conjugacy classes of $1$--sided simple closed curves pairwise intersecting minimally. Notice that, thanks to Remark \ref{2sided}, each face in $\Upsilon^{(2)}$ also corresponds to a $2$--sided curve.

\begin{figure}
[hbt] \centering
\includegraphics[height=2 cm]{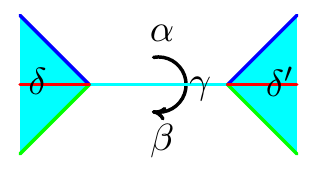}
\caption{The edge $e\leftrightarrow (\alpha, \beta, \gamma;\delta, \delta')$. (The coloring will be explained in the next section.)}
\label{fig:edge}
\end{figure}

We use greek letters $\a, \b, \g, \ldots \a_i,\dots$ to denote the elements of $\Upsilon^{(3)}$, while for the elements in$\Upsilon^{(2)}$ we use $\xi$, $\xi_{\a, \b}, \dots$. For an edge $e \in \Upsilon^{(1)}$, we also use the notation $e\leftrightarrow (\a,\b,\g; \d, \d')$ to indicate that $e=\a \cap \b \cap \g$ and $e\cap \d$ and $e \cap \d'$ are the endpoints of $e$; see Figure \ref{fig:edge}.

\subsection{The coloring of the tree}

We choose a coloring of the regions and edges, namely a map $\mathcal{C} \co \Upsilon^{(3)} \cup \Upsilon^{(1)}  \to \{ 1, 2, 3, 4 \}$ such that for any edge $e\leftrightarrow (a,\b,\g; \d, \d')$ we have $\mathcal{C}(e)=\mathcal{C}(\d) = \mathcal{C}(\d')$ and such that $\mathcal{C} (e)$, $\mathcal{C}(\a)$ , $\mathcal{C}(\b)$, and $\mathcal{C}(\g)$ are all different. The coloring is completely determined by a coloring of the four regions around any specific vertex, and hence is unique up to a permutation of the set $\{1, 2, 3,4\}$. We denote by $\Upsilon^{(3)}_i$ the set of complementary regions with color $i$, and by $\Upsilon^{(1)}_i$ the set of edges with color $i$. In Figure \ref{fig:edge} the edge $e =(\alpha, \beta, \gamma;\delta, \delta') \in \Upsilon^{(1)}_4$ is drawn. (We didn't color the three regions $\alpha, \beta, \gamma$ around $e$.)

In the following, when $\a, \b, \g, \d$ are complementary regions around a vertex, we will use the convention that $\alpha \in \Upsilon^{(3)}_1$, $\beta \in \Upsilon^{(3)}_2$, $\gamma \in \Upsilon^{(3)}_3$, and $\delta \in \Upsilon^{(3)}_4$.

\begin{Remark}\label{bi-colored}
  Note that there is a one-to-one correspondence between faces in $\Upsilon^{(2)}$ and so called bi-colored geodesics in $\Upsilon^{(1)}$, namely maximal subtrees whose edges are of two colors. This set of edge is the boundary of the face.
  
  In the same way, there is a one-to-one correspondence between faces in $\Upsilon^{(3)}$ and tri-colored subtree in $\Upsilon^{(1)}$.
\end{Remark}

\subsection{$\Uom$--Markoff quads}

For a triple $\Uom = (x,y,z) \in \C^3$, a $\Uom$-{\it Markoff quad} is an ordered quadruple $(a_1, a_2, a_3, a_4)$ of complex numbers satisfying the $\Uom$--Markoff equation: 
\begin{equation}\label{eqn:vertex}
	\begin{split}
		a_1^2+a_2^2 +a_3^2+a_4^2+a_1a_2a_3a_4 = & \l_{12}(a_1a_2+a_3a_4)+\l_{23}(a_2a_3+a_1a_4)+\l_{13}(a_1a_3+a_2a_4) \\
			&  +4-\l_{12}^2-\l_{23}^2-\l_{13}^2-\l_{12}\l_{23}\l_{13},
	\end{split}
\end{equation}
or, equivalently,
\begin{equation}
	\sum_{i=1}^4 a_i^2 + \prod_{i=1}^4 a_i =   4  - \l_{12} \l_{23} \l_{13} +  \sum_{i<j = 1}^4 \l_{ij} \left( a_i a_j - \frac{\l_{ij}}{2} \right), 
\end{equation}
where
\begin{equation}\label{lambda}
  \begin{aligned}
     \l_{12}=\l_{34} = x,\\ 
     \l_{23} = \l_{14} = y,\\ 
     \l_{13} = \l_{24} = z.
  \end{aligned}
\end{equation}

It is easily verified that, if $(a_1, a_2, a_3, a_4)$ is a $\Uom$--Markoff quad, then the quad obtained by replacing $a_i$ by $a_i'$ is also a $\Uom$--Markoff quad, where
\begin{equation}\label{eqn:elemoper}
a_i' = \sum_{j \neq i} \l_{ij} a_j - \prod_{j\neq i} a_j - a_i.
\end{equation}

Since equation \eqref{vertex} is not symmetric in the variables $a_1, \ldots, a_4$, then the permutations quads are not $\Uom$--Markoff quads. This is different from the situation analyzed in Huang--Norbury \cite{hua_sim}, where the equation defining the case of parabolic boundary components is symmetric, and much simpler to manipulate.

\subsection{Relation with $\Uom$-representations}

We start by the following basic remark

\begin{Remark}\label{correspondence}
  A representation $\rho$ is in $\mathfrak{X}_{\Uom}$, where $\Uom \in \C^3$, if and only if there exists a set of generators $\a, \b, \g$ for $F_3$ such that $(\Tr(\rho(\a)), \Tr(\rho(\b)), \Tr(\rho(\g)), \Tr(\rho(\a\b\g)))$ is a $\Uom$--Markoff quad, with $\Uom = (\Tr (\rho (\a\b)), \Tr (\rho (\b\g)), \Tr (\rho (\a\g)))$.
\end{Remark}

Using this correspondence, we can now see that the elementary operations defined in \eqref{eqn:elemoper} are intimately related with the action of the mapping class group on the character variety, see Equation (\ref{eqn:theta_a}).

\subsection{$\Uom$--Markoff maps}

A $\Uom$-{\it Markoff map} is a function $\psi \co \Upsilon^{(3)} \cup \Upsilon^{(2)} \to \C$ such that the following properties hold:
\begin{itemize}
\item[(i)] \underline{{\em Vertex Equation}}: $\forall v = (\a_1, \ldots, \a_4) \in \Upsilon^{(0)}$ (i.e the $\a_i \in \Upsilon^{(3)}_i$ are the four regions meeting the vertex $v$), the quad $(\psi(\a_1), \psi(\a_2), \psi(\a_3), \psi(\a_4))$ is a $\Uom$--Markoff quad;
\item[(ii)] \underline{{\em Edge Equation}}: $\forall e = (\a_i, \a_j, \a_k; \a_l, \a_l') \in \Upsilon^{(1)}_l$,  $$\psi(\a_l)+\psi(\a_l') = \l_{il}\psi(\a_i) + \l_{jl}\psi(\a_j) + \l_{kl}\psi(\a_k) - \psi(a_i) \psi(a_j) \psi(a_k);$$
\item[(iii)] \underline{{\em Face Equation}}: $\forall f = (\a_i, \a_j) \in \Upsilon^{(2)}$, $$\psi(\a_i)\psi(\a_j) = \l_{ij} + \psi(\xi_{\a_i, \a_j}).$$
\end{itemize}
where $\l_{ij}$ are defined in Equation \ref{lambda}.

We shall use ${\bf \Psi}_{\Uom}$ to denote the set of all $\Uom$--Markoff maps and lower case letters to denote the $\psi$ values of the regions, that is, $\psi(\a_i)=a_i$.

Note that if the Vertex Equation (i) is satisfied at one vertex, then the Edge Equation (ii) guarantees that the Vertex Equation (i) is in fact satisfied at every vertex. This is related to the fact that the Edge Equations, arise from the action of the (pure) mapping class group $\MCG(N)$ of the three-holed projective plane $N$, which preserves the boundary traces, and hence the relative character variety $\mathfrak{X}_{\Uom}$. This tells us the following:

\begin{Remark}\label{phi_ut}
  There exists a bijective correspondence between $\Uom$--Markoff maps and $\Uom$--Markoff quads. Hence, using Remark \ref{correspondence}, there exists a bijective correspondence between the set ${\bf \Psi}_{\Uom}$ of $\Uom$--Markoff maps and the $\Uom$--relative character variety $\mathfrak{X}_{\Uom}$.
\end{Remark}

Given a Markoff map $\psi \in {\bf \Psi}_{\Uom}$, we now introduce a {\em secondary function} $\sigma = \sigma_\psi\co \Upsilon^{(2)} \rightarrow \C$ as follows. If $\xi = (\a_i , \a_j) \in \Upsilon^{(2)}$, then 
$$\sigma(\xi) = (a_i^2+a_j^2+\l_{ij}^2-a_i a_j \l_{ij} - 4)(\l_{ik}^2+\l_{jk}^2+\psi(\xi)^2 - \l_{ik}\l_{jk}\psi(\xi) - 4),$$ where $k \neq i$ and $k \neq j$. (Note that any of the two possible choices for $k$ give the same function.) The zeroes of this function are related to representations which, restricted to a certain subsurface, are reducible, as we will   explain in more details in Remark \ref{rk:sigma}.And in   Section \ref{s:face}, we will describe the behavior of the  markoff map $\psi$ when for a certain face $\xi \in \Upsilon^{(2)}$ we have $\sigma(\xi) = 0$. 

\subsection{Orientation on $\Upsilon^{(1)}$} 

As in the previous papers \cite{bow_mar, tan_gen, mal_ont}, a $\Uom$--Markoff map $\psi \in {\bf \Psi}_{\Uom}$ determines an orientation on the $1$--skeleton $\Upsilon^{(1)}$ as follows. Suppose $e = (\a,\b, \g; \d, \d')$. If $|d|>|d'|$, then the arrow on $e$ points towards $\d'$, while if $|d'|>|d|$, then the arrow on $e$ points towards $\d$. If $|d|=|d'|$, then we choose the orientation of $e$ arbitrarily. The choice does not affect the arguments in the latter part of this paper.

A vertex with all four arrows pointing towards it is called a {\em sink}, while one where all the four arrows point away from it is called a {\em source}. When three arrows point towards it and one away is called a {\em merge}, and in all the other cases, namely, a vertex with at least two arrows pointing away from it, is called a {\em (generalized) fork}.

\section{Analysis of $\Uom$-Markoff maps}\label{s:analysis}

In this section we define the so-called BQ-conditions for Markoff maps and analyze the behavior of maps satisfying them. In particular, for any Markoff maps we will define an orientation on the 1-skeleton of $\Upsilon$ and an attracting subtree, which we will use in the following section to prove Theorem \ref{Main}. In fact, the attracting subtree will be finite if and only if the Markoff maps satisfies the BQ-conditions.

\subsection{BQ-conditions for Markoff maps}\label{s:BQ}

Let  $\Uom = (x,y,z) \in \C^3 $ and denote 
\begin{equation} M = M(\Uom) = \max \{ |x| , |y| , |z| \}. \end{equation}
Given $\psi \in {\bf \Psi}_{\Uom}$, and $K > 0$,  we define the subsets:
\begin{align*}
	\Upsilon_\psi^{(3)} (K) & = \left\{ \alpha \in \Upsilon^{(3)} \, \mid \, | \psi (\alpha) | < K \right\}; \\
	\Upsilon_\psi^{(2)} (K) & = \left\{ \xi_{\a , \b} \in \Upsilon^{(2)} \, \mid \, \a \mbox{ or } \b \in \Upsilon_\psi^{(3)} (K), \mbox{ and }| \psi (\xi) | < K^2+M \right\}.
\end{align*}

Note that if two adjacent regions $\alpha, \beta$ are in $\Upsilon_\psi^{(3)} (K)$, then the face $\xi_{\alpha, \beta}$ is in $\Upsilon_\psi^{(2)} (K)$, which makes the following lemmas easier to state.

\begin{Definition}[BQ set $({\bf \Psi}_{\Uom})_Q$]

The {\em Bowditch set} $({\bf \Psi}_{\Uom})_Q$ coincides with the set of $\Uom$-Markov maps $\psi \in {\bf \Psi}_{\Uom}$ such that the following conditions hold:
    \begin{enumerate}
    	\item [{(BQ1)}] For all $\xi \in \Upsilon^{(2)}$, we have $\psi (\xi) \notin [-2 , 2]$.
    	\item [{(BQ3)}] For all $\xi \in \Upsilon^{(2)}$, we have $\sigma (\xi) \neq 0$.
	    \item [{(BQ4)}] The set $\Upsilon^{(2)}_{\psi}(2+M)$ is finite.
    \end{enumerate}
\end{Definition}

\begin{Remark}
   By Remark \ref{phi_ut} we have an identification between the set ${\bf \Psi}_{\Uom}$ of $\Uom$--Markoff maps and the $\Uom$--relative character variety $\mathfrak{X}_{\Uom}$. So we can also define the set $(\mathfrak{X}_{\Uom})_Q$ of {\em Bowditch representations}, that will be referred to as the {\em Bowditch set}.
\end{Remark}

\begin{Remark}
	The definition of the BQ-condition for Markoff maps is different from the one given in the Introduction. In particular, condition $(BQ3)$ seems new, and condition $(BQ4)$ is much weaker than condition $(BQ2)$. However, in Proposition \ref{pro:equiv_def} we will see that the two definitions are, in fact, equivalent. 
\end{Remark}

\subsection{Connectedness}

We start with the following key result.

\begin{Lemma}[Fork Lemma]\label{lem:fork}
	Let $\psi$ be a $\Uom$-Markoff map and $v = (\a, \b, \g, \d) \in \Upsilon^{(0)}$. Suppose that two arrows induced by $\psi$ points away from $v$. Then at least one of the six faces $\xi_{\a,\b}, \xi_{\a,\g}, \xi_{\a,\d}, \xi_{\b,\g}, \xi_{\b,\d}, \xi_{\g,\d}$ passing through $v$ is in $\Upsilon_\psi^{(2)} \left( 2+M \right)$.
\end{Lemma}
\begin{proof}
Suppose, without loss of generality, that the outgoing arrows are in the direction of $\g$ and $\d$. The edge relations give
\begin{align*}
	 c + c' = (x-ab)d +yb+za \\
	 d+d' = (x-ab)c + ya + zb,
\end{align*}
and the directions of arrows give
	$$|c| > |c'| \hspace{0.5cm} \mbox{and} \hspace{0.5cm} |d| > |d'|.$$
	From these, we get the two inequalities
	\begin{align*}
		 2|c| \geq &  |ab-x| |d| - (|yb| + |za|), \\
		 2|d| \geq &  |ab-x| |c| - (|ya| + |zb|),
	\end{align*}
and adding both inequalities, we get:
	$$2(|c|+|d|) \geq |ab-x|(|c|+|d|)-(|a|+|b|)(|y|+|z|).$$

\vskip 4pt
First, let's prove that one of the region $\a, \b , \g , \d$ is in $\Upsilon_\psi^{(3)} (2+M)$.
	
	By contradiction, assume that $|a|, |b|, |c| , |d| \geq 2+M$. Then we have that $|ab-x| \geq |a|+|b|$, and also that $|c|+|d|-|y|-|z|\geq 4$. So 
	\begin{align*} 
		2(|c|+|d|) & \geq |ab-x|(|c|+|d|-|y|-|z|)\\
		|ab-x| & \leq \dfrac{2(|c|+|d|)}{|c|+|d|-|y|-|z|} \\
			& \leq 2 + \dfrac{2(|y|+|z|)}{|c|+|d|-|y|-|z|} \\
			& \leq 2 + \frac{4M}{4} \leq 2+M.
	\end{align*}
On the other hand, if $|ab-x| \leq 2+M$, then $|ab| \leq 2+2M$, and hence $$\min\{ |a| , |b| \} \leq \sqrt{2+2M} < 2+M,$$ which gives the contradiction proving the claim.
	
  \vskip 4pt
	
	Now, we proceed to show that one of the face is in $\Upsilon_\psi^{(2)} (2+M)$. 
	
	If two regions among $\a, \b, \g$ and $\d$ are in $\Upsilon_\psi^{(3)} (2+M)$ then the face at the intersection of these two regions is directly in $\Upsilon_\psi^{(2)} (2+M)$. For example, if $|a|, |d| < 2+M$, then 
$$|ad-y| < |ad|+|z| < (2+M)(2+M) + M < (2+M)^2 + M.$$
So in this case, $\xi_{\a ,\d} \in \Upsilon_\psi^{(2)} (K)$.
	
	So we can assume that only one region is in $\Upsilon_\psi^{(3)} (2+M)$. As the regions $\a , \b$ and $\g, \d$ play a different role, we have to distinguish two cases.
	
	\underline{\textit{Case 1}:} $\g$ or $\d$ is in $\Upsilon_\psi^{(3)} (2+M)$.
	
	Without loss of generality, consider the case $|d|<2+M$. Assume by contradiction that $|a| \geq |b| \geq |c| \geq 2+ M$. Then the relation $abc = ay+bz+cx+d+d'$ induces the inequality: 
$$(2+M)^2 |a| \leq |abc| \leq |ya|+|bz|+|cx| + 2|d| < 3M|a|+4+2M.$$
Hence we get $(4+M+M^2)|a|  < 4+2M$, and so $|a|  < \frac{1}{2} (M+2)$. This gives a contradiction. So at least one of $|a|, |b|$ or $|c|$ is less than $2+M$.

\underline{\textit{Case 2}:} $\a$ or $\b$ is in $\Upsilon_\psi^{(3)} (2+M)$.

Without loss of generality, consider the case $|a|<2+M$. Assume that $\xi_{\a , \g}$ and $\xi_{\a , \d}$ are not in $\Upsilon_\psi^{(2)} (2+M)$, so that $|ac-z| \geq (2+M)^2+M$ and $|ad-y| \geq (2+M)^2+M$. We will prove that, in this case, $\xi_{\a, \b}$ is in $\Upsilon_\psi^{(2)} (2+M)$.  

As $|c|>|c'|$, and $|d| > |d'|$, we have 
\begin{align*}
|ac'-z| < |ac-z|+2M \\
|ad'-y| < |ad-y| + 2M.
\end{align*}
because $|ac'-z| < |ac'|+ |z| < |ac| + |z| < |ac'-z| + 2|z|.$
Hence we have 
	\begin{align*}
		2|ac-z| & > |ac-z| + |ac' - z| - 2M \\
			&> |a(c+c')-2z | - 2M \\
			&> |a(-abd+xd+yb+za) - 2z | - 2M \\
			&> |-(ab-x)(ad-y) + (a^2-2)z + xy | - 2M \\
			&> |ab-x| \, |ad-y| - |(a^2-2)z + xy| - 2M \\
			&> |ab-x| \, |ad-y| - (M^3+5M^2+8M).
	\end{align*}
	Similarly we have $2|ad-y|> |ab-x| \, |ac-z| - (M^3+5M^2+8M)$. 
	
	Adding the two inequalities, we obtain
	$$|ab-x| < 2 + \dfrac{2(M^3+5M^2+8M)}{|ac-z| + |ad-y|}.$$
	As $|ac-z| + |ad-y|> 2 ((2+M)^2 + M)$, we get 
	$$|ab-x| < 2+\dfrac{2 (M ((2+M)^2+M) + 4M)}{2 ((2+M)^2 + M)} < (2+M)^2+M.$$
	So $\xi_{\a , \b} \in \Upsilon_\psi^{(2)} (2+M)$, as desired.
\end{proof}

\begin{Lemma}\label{lem:conn_3}
   The set $\Upsilon_\psi^{(3)} (K)$ is connected, for all $K \geq 2+M$.
\end{Lemma}

\begin{proof}
	By contradiction, suppose that $\Upsilon_\psi^{(3)} (K)$ is not connected. Then there exists two regions $\d , \d' \in \Upsilon_\psi^{(3)} (K)$ at distance $m\geq 1$ from each other, such that $\d$ and $\d'$ cannot be connected within $\Upsilon_\psi^{(3)} (K)$, and such that the distance $m$ is minimal.
	
	\underline{\textit{Case 1}:} $m=1$. 
	
	The regions $\d$ and $\d'$ are connected by an edge $(\a,\b,\g)$ so that we have $d+d' = -abc + ay+bz+cx$. By hypothesis, $|a|, |b|, |c| > K$ and hence: 
	
	\begin{align*}
		|abc| &= |K+ (|a|-K)|\,  |K+ (|b|-K)  |\,  |K+(|c|-K) |\\
			&> K^3 + K^2 ((|a|-K)+(|b|-K)+(|c|-K))  \\
			& >K^3+ K^2 (|a|+|b|+|c|-3K).
	\end{align*}

	On the other hand, 
	\begin{align*}
		|abc| & <|ay|+|bz|+|cx|+|d|+|d'| \\
			& < M(|a|+|b|+|c|)+2K \\
			& < 3K^2-4K + (K-2)(|a|+|b|+|c|-3K) \\
			& < K(3K-4) + K ((|a|+|b|+|c|-3K)\\
			& < K^3 + K (|a|+|b|+|c|-3K),
	\end{align*}
	which gives a contradiction. So $m \neq 1$.

	\underline{\textit{Case 2}:} $m>1$.
	
	Consider the sequence of edges $(e_i)_{1\leq i \leq m}$ going from $\d$ to $\d'$. By minimality of $m$, the arrows $e_1$ and $e_m$ point towards the regions $\d$ and $\d'$ respectively. Hence, one of the vertices along the sequence of edges is a fork. So, using the Fork Lemma, one of the regions neighboring the fork is in $\Upsilon_\psi^{(3)} (K)$, and this contradicts the minimality of $m$.
	\end{proof}

\begin{Lemma}\label{conn_2}
	The set $\Upsilon_\psi^{(2)} (K)$ is edge-connected, for all $K \geq 2+M$.
\end{Lemma}	

\begin{proof}
	We need to prove that, if two faces $\xi$ and $\xi'$ are in $\Upsilon_\psi^{(2)} (K)$, then we can find a sequence of edges such that, for each edge, one of the faces touching the edge is in $\Upsilon_\psi^{(2)} (K)$.
	
	If $\xi = (\a , \b)$ is  in $\Upsilon_\psi^{(2)} (K)$, then one of the region $\a$ or $\b$ is in $\Upsilon_\psi^{(3)} (K)$ . Likewise, if $\xi' = (\a' , \b') \in \Upsilon_\psi^{(2)} (K)$, then one of the region $\a'$ or $\b'$ is in $\Upsilon_\psi^{(3)} (K)$. 
	
	As $K\geq 2+M$, the set $\Upsilon_\psi^{(3)} (K)$ is connected. Hence we can find a sequence of regions 
	$$\a = \g_0,  \g_1 , \g_2 , \dots , \g_n = \a'  \mbox{   in   } \Upsilon_\psi^{(3)} (K)$$
	 connecting $\a$ to $\a'$. This gives in turn a sequence of faces $\xi_k = (\g_k , \g_{k+1})$. \\ As $|\psi(\xi_k)| \leq  |\psi(\g_k) \psi(\g_{k+1})|+M \leq K^2+M$, it is clear that $\xi_k$ is in $\Upsilon_\psi^{(2)} (K)$.
\end{proof}

\subsection{Escaping rays}\label{s:neighbor}

Let $P$ be a geodesic arc in the tree $\Upsilon$, starting at vertex $v_0 \in \Upsilon^{(0)}$ and consisting of edges $e_n$ joining $v_n$ to $v_{n+1}$. We say that such an infinite geodesic is an {\em escaping ray} if each edge $e_n$ is directed from $v_n$ towards $v_{n+1}$. First we will consider the case of an escaping ray laying in the boundary of a face $\xi = (\a , \b) \in \Upsilon^{(2)}$, and then we will describe the general case.

\subsubsection{Neighbors of a face}\label{s:face}

Each face $\xi = (\a , \b) \in \Upsilon^{(2)}$ is a bi-infinite path consisting of edges of the form $(\a , \b , \g_n)$ alternating with edges $(\a , \b , \d_n)$, where $(\a, \b , \g_n, \d_n)$ and $(\a, \b , \g_{n+1}, \d_n)$ are vertices in $\Upsilon^{(0)}$. We say that $\g_n$ and $\d_n$ are the {\em neighboring regions} to the face $\xi$. The edge relations on two consecutive edges give:
\begin{align*}
	c_{n+1} & = - (ab-x) d_n  - c_n + (za+yb);  \\
	d_{n+1} & = - (ab-x) c_{n+1} - d_n + (zb+ya)  \\
		& = (ab-x) c_n + ((ab-x)^2-1)d_n + (zb+ya-(ab-x)(za+yb)).
\end{align*}

We can reformulate these equations in terms of matrices:
$$\begin{pmatrix} c_{n+1} \\ d_{n+1} \end{pmatrix} = \begin{pmatrix} -1 & -(ab-x) \\ (ab-x) & (ab-x)^2-1 \end{pmatrix} \cdot \begin{pmatrix} c_{n} \\ d_{n} \end{pmatrix} + \begin{pmatrix} za+yb \\ zb+ya-(ab-x)(za+yb) \end{pmatrix}.$$

Note that the setting is similar to the situation for the four-holed sphere in \cite{mal_ont}, up to a change of variables. See also Appendix \ref{app:functionH}.

Let $\lambda \in \C$ such that $\lambda+\lambda^{-1} = (ab-x)^2 - 2$. It corresponds to the two eigenvalues of the matrix $\begin{pmatrix} -1 & -(ab-x) \\ (ab-x) & (ab-x)^2-1 \end{pmatrix}$. Note that $|\lambda| = 1$ if and only if $(ab-x) \in [-2 , 2]$. Take $\Lambda$ to be a square root of $\lambda$.

If $(ab-x)\notin \{ - 2 , 2 \}$ then we can express the sequence $c_n$ and $d_n$ as:
\begin{align*}
	c_n & = A \Lambda^{2n} + B \Lambda^{-2n} + \eta(ab-x), \\
	d_n & = - (A \Lambda^{2n+1} + B \Lambda^{-2n-1}) + \zeta(ab-x),
\end{align*}
where $\eta$ and $\zeta$ are two complex functions with parameters $a, b$ defined by:

\begin{align*}
	\eta(t) &= \dfrac{1}{4-t^2} \left( 2 (za+yb)-t(zb+ya)\right), \\
	\zeta(t) &= \dfrac{1}{4-t^2} \left( 2 (zb+ya)-t(za+yb)\right).
\end{align*}

Note that $\eta(ab-x)$ and $\zeta (ab-x)$ are the coordinates of the center of the quadric in coordinates $(c,d)$ defined by the vertex relation \eqref{vertex} [with parameters $(a,b,x,y,z)$].

The product $AB$ is given by the following: 
$$AB = \dfrac{1}{(4-(ab-x)^2)^2} (a^2+b^2+x^2-abx-4)((ab-x)^2+y^2+z^2-yz(ab-x)-4 ).$$
Using the secondary function $\sigma$, it can be written more concisely as: 
$$AB = \dfrac{\sigma (\xi)}{4-(\psi(\xi))^2}.$$

From this discussion, we deduce the following result:

\begin{Lemma}\label{lem:neighbors}
	With the notations introduced above, we have
	\begin{enumerate}
		\item If $\psi (\xi) \in (-2 , 2)$, then $|c_n|$ and $|d_n|$ remain bounded.
		\item If $\psi (\xi) \in \{-2 , 2\}$, then $|c_n|$ and $|d_n|$ grow at most quadratically.
		\item If $\psi (\xi) \notin [-2 , 2 ]$, and $\sigma (\xi) \neq 0$, then $|c_n|$ and $|d_n|$ grows exponentially as $n\rightarrow +\infty$ and as $n \rightarrow - \infty$.
		\item If $\psi (\xi) \notin [-2 , 2 ]$, and $\sigma (\xi) = 0$, then $(c_n, d_n)$  converges to $(\eta(ab-x) , \zeta(ab-x))$ when $n \rightarrow +\infty$ or $n \rightarrow -\infty$.
	\end{enumerate}
\end{Lemma}

\subsubsection{General escaping ray}\label{s:general}

Now we consider the general case.

\begin{Lemma}\label{lem:escaping}
	Suppose that $\{ e_n \}_{n \in \N}$ is an escaping ray. Then:
	\begin{itemize}
		\item either the ray is eventually contained in some face $\xi \in \Upsilon^{(2)}$ such that $\psi(\xi) \in [-2 , 2]$, or $\sigma (\xi) = 0$, 
		\item or the ray meets infinitely many elements $\g \in \Upsilon_{\psi} (2+M)$.
	\end{itemize}
\end{Lemma}
\begin{proof}
	First, suppose that there exists $n_0 \in \mathbb{N}$, and a face $\xi \in \Upsilon^{(2)}$ such that, for all $n \geq n_0$, the edges $e_n$ are contained in $\xi$. As the path is descending, it means that the values of the regions meeting $\xi$ at the edges $e_n$ stay bounded. Hence, from Lemma \ref{lem:neighbors}, we infer that, either $\psi (\xi) \in [-2 , 2]$, or $\sigma (\xi) = 0$. 
  
  \vskip 5pt
	
	Now suppose that we are not in the first case. Then we prove the following.
	
	\textbf{Claim:} There exists $n_0 \in \N$ such that $e_{n_0}$ is contained in a face $\xi_0 \in \Upsilon^{(2)} (2+M)$. 
	
	The proof is very similar to the one of the Fork Lemma \ref{lem:fork}. Let $\varepsilon > 0$. We note $\a_i , \b_i , \g_i , \d_i$ the sequences of neighboring regions around the edge $e_i$. At least two sequences among $(|a_i|) , (|b_i|) , (|c_i|) , (|d_i|)$ are infinite, decreasing and bounded below. So for $n$ large enough, we have two consecutive edges $e_n = (\a , \b , \g ; \d' \rightarrow \d )$ and $e_{n+1} = (\a , \b , \d , \g \rightarrow \g')$ with a common face $\xi_{\a , \b}$, and such that 
	$$|d | \leq |d' | \leq |d | + \varepsilon \mbox{   and   } |c'| \leq |c |.$$
	Using the same argument as in Lemma \ref{lem:fork}, we can prove that, if $|a| , |b| , |c| , |d| > 2+M$, then $|ab| \leq 2 + 2M + \varepsilon$. For $\varepsilon$ small enough, we get that 
	$$\min \{ |a| , |b| \} \leq \sqrt{2+2M+\varepsilon} \leq 2+M,$$
	which gives a contradiction. So one of the regions around the vertex $v_n$ is in $\Upsilon^{(3)}_{\psi} (2+M)$. In turn, this gives the existence of a face $\xi_0$ in $\Upsilon^{(2)}_{\psi} (2+M)$ around the edge $e_n$, which proves the claim.
	
\vskip 5pt
		
	As we are not in the first case, it means that there exists $n_1 > n_0$ such that $e_{n_1} \notin \xi_0$. Let $v_1$ be the head of the arrow $e_{n_1}$ and consider the escaping ray starting at $v_1$. By the same reasoning, the ray will meet a face $\xi_1 \in  \Upsilon^{(2)} (K)$.  By induction, this proves that the initial ray will eventually meet an infinite number of faces in $\Upsilon^{(2)} (2+M)$.
\end{proof}

We will also need the following result, giving a necessary condition for a given element $\xi \in \Upsilon^{(2)}$ to satisfy $\sigma (\xi) = 0.$

\begin{Lemma}
	Suppose that we have $\psi \in {\bf \Psi}_{\Uom}$ and $\xi \in \Upsilon^{(2)}$ such that $\sigma (\xi) = 0$, and let $K > 2+M$. Then at least one of the following is true:
	\begin{itemize}
		\item The edge $\xi \in \Upsilon^{(2)} (K)$.
		\item The set $\Upsilon^{(2)} (K)$ is infinite.
	\end{itemize}
\end{Lemma}
\begin{proof}
	If $\psi (\xi) \in [-2 , 2]$, then the first condition is satisfied. So suppose that $\psi(\xi) \notin [-2 , 2]$. Let $\g_n$ and $\d_n$ be the two sequence of neighboring regions around the face $\xi$. From the previous Lemma \ref{lem:neighbors} we can see that the sequence $(c_n, d_n)$  converges to a certain point in $\C^2$ when $n$ goes to infinity. 
	
	Let $\varepsilon >0$. For $n$ large enough, the successive values of $|c_n|$ and of $|d_n|$ are as close as we want. We can then use the same proof as in the Fork Lemma to prove that for $n$ large enough, at the vertex $v_n$, which is the intersection  $\xi \cap \g_n \cap \d_n$, one of the faces containing $v_n$ is in $\Upsilon^{(2)} (2+M+\varepsilon)$ and hence in $\Upsilon^{(2)} (K)$, if $\varepsilon$ is small enough.
	
	The face $\xi$ is the only face that contains more than one vertex $v_n$. So either, $\xi$ is in $\Upsilon^{(2)} (K)$, or there exists an infinite number of faces in $\Upsilon^{(2)} (K)$.
\end{proof}

\subsection{Attracting subtree}\label{sec:att_tree}

In this section for all $\psi \in {\bf \Psi}_{\Uom}$ and for all $K \geq 2+M$ we define an attracting subtree $T_\psi (K)$. In order to do that, first we construct, for all faces $\xi \in \Upsilon^{(2)}$, an attracting subarc in the boundary of $\xi$.

\subsubsection{Attracting arc $J_\psi (K, \xi)$}

Let $K \geq 2+M$. For each map $\psi \in {\bf \Psi}_{\Uom}$ and for each face $\xi \in \Upsilon^{(2)} (K)$, we will construct explicitly a connected non-empty subarc $J = J_\psi (K, \xi)$ such that the following conditions are satisfied:
\begin{enumerate}
	\item Every edge in $\xi$ that is not in $J$ points towards $J$.
	\item For all $\xi \in \Upsilon^{(2)} (K)$, with $K \geq 2+M$,  if $X \in \Upsilon^{(k)}(K)$, with $k = 2, 3$, and $\xi \cap X = e$, then $e \in J$.
\end{enumerate}

From Lemma \ref{lem:neighbors}, we can define a function $$H_{\psi} \co \Upsilon^{(2)} \to \R \cup \{\infty\},$$ as follows:
\begin{itemize}
  \item If $\xi \in \Upsilon^{(2)}$, such that $\sigma_\psi(\xi) = 0$ or $\psi(\xi) \in [-2 , 2]$, then we define $H_{\psi}(\xi) = \infty$.
  \item If $\xi = \xi_{\a, \b} \in \Upsilon^{(2)}$, such that $\sigma_\psi(\xi) \neq 0$ and $\psi(\xi) \notin [-2 , 2]$, we have that, if $\g_n$ and $\d_n$ are the sequences of neighboring regions around $\xi$, then there exists integers $n_1 , n_2$ such that:
  \begin{itemize}
  	\item $|c_n| \leq H_{\psi} (\xi)$ and $ |d_n| \leq H_{\psi} (\xi)$ if and only if $ n_1 \leq n \leq n_2$;
  	\item $|c_n|$ and $|d_n|$ are monotonically decreasing for $n < n_1$, and increasing for $n > n_2$.
  \end{itemize}
\end{itemize}
(See Section \ref{s:neighbor} for the definition of `neighboring regions'.)
 
As the explicit expression of the function $H_{\psi}$ is not relevant in the following proofs, we will defer its definition to the Appendix \ref{app:functionH}.

\begin{Remark}
  Note that, for any fixed face $\xi\in \Upsilon^{(2)}$, the function $$H_{\cdot}(\xi) \co {\bf \Psi}_{\Uom} \to \R \cup \{\infty\}$$ defined by $\psi \mapsto H_\psi(\xi)$ is continuous.
\end{Remark}

The arc formed by the union of the edges $e_n$ for $n_1 \leq n \leq n_2$ satisfies property (1). In order for it to satisfy property (2), we need to slightly modify the function $H$ into $$H_{\psi}^{\ast} \co \Upsilon^{(2)}\times \mathbb{R} \to \R \cup \{\infty\},$$ as follows:
$$H_{\psi}^{\ast} (\xi, K) := \max \left\{ H_{\psi} (\xi), \frac{K^2+2M}{\min\{|a|, |b|\}} \right\}.$$

Now we define:
$$J_\psi (K, \xi) := \displaystyle\bigcup_{\g \in \Upsilon^{(3)}( H_{\psi}^{\ast} (\xi, K))} \g \cap \xi.$$

This is a subset of the edges in the boundary of $\xi$ such that the regions corresponding to each edge have image less than $H_{\psi}^{\ast} (\xi, K)$. 
Note that, if a face $\xi \in \Upsilon^{(2)}$, satisfies $\sigma_\psi(\xi) = 0$ or $\psi(\xi) \in [-2 , 2]$, then the subset $J_\psi (H , \xi)$ is the entire face $\xi$.

Now it is clear that the arc $J_\psi (K, \xi)$ constructed by this procedure satisfies conditions $(1)$ and $(2)$ above.

\subsubsection{Attracting subtree $T_\psi (K)$}

The previous discussion guarantees the existence of an attracting subarc for all faces $\xi \in \Upsilon^{(2)} (K)$. Hence, we can construct the set:
$$T_\psi (K) = \bigcup_{\xi \in \Upsilon^{(2)} (K)} J_\psi (\xi, K).$$
This set is a union of edges, and we can prove the following result.

\begin{Proposition}
	The set $T_\psi (K)$ is connected and $\psi$-attracting.
\end{Proposition}

\begin{proof}
	Denote $T = T_\psi (K)$, and let $e$ and $e'$ be two edges in $T$. From the construction of $T$, there are two faces $\xi , \xi' \in \Upsilon^{(2)} (K)$ such that $e \in \xi$ and $e' \in \xi'$. By edge connectedness of $\Upsilon^{(2)} (K)$ we can find a sequence of faces $\xi = \xi _0 , \xi_1 , \dots , \xi_n = \xi'$ such that $\xi_k \cap \xi_{k+1} = e_k \neq \emptyset$. By property (2), each edge $e_k$ is in $T_\psi(K)$. As $J_\psi (\xi_k, K)$ is connected, the edge $e_{k-1}$ and $e_k$ are connected inside $\xi_k$. So, we get that $e$ and $e'$ are connected inside $T$.
	\vskip 5pt
	Now we prove that $T$ is $\psi$-attracting, using the following claim. 
	
	\textit{Claim:} The arrows in the circular neighborhood $C(T)$ around $T$ point towards $T$.
	
	 Indeed, suppose that we have a vertex $v = (\a ,\b ,\g ,\d) \in T$ and an edge $e \notin T$ pointing outward. One of the face containing $v$ is in $\Upsilon^{(2)}(K)$, for example $\xi_{\a, \b}$. By the property (1) of the arc $J_\psi (\xi_{\a , \b }, K)$, it is clear that $e$ cannot be contained in $\xi_{\a , \b}$. So we can assume, without loss of generality, that $e = (\b , \g , \d)$.
	
	If $|a| < K$, then we have $|a'|<K$. Then by connectedness of $\Upsilon^{(3)}(K)$, one of the region $\b , \g$ or $\d$ is also in $\Upsilon^{(3)}(K)$. So one of the faces $\xi_{\a' , \b}$, $\xi_{\a' , \g}$ or $\xi_{\a' , \d}$ is in $\Upsilon^{(2)}(K)$. This contradicts the connectedness of $T$.
	
	If $|a|> K$, then we have $|b|<K$. The same kind of inequality gives:
	$$|bc-y| \, |bd - z | < 2 |ab| + |xy| + |x| \, |b|^2 < (2+M)^3 + M^2 + 4M < (2+M)^4,$$
	which proves that one of the face $\xi_{\b, \g}$ or $\xi_{\b , \d}$ is in $\Upsilon^{(2)}(K)$, again contradicting the connectedness of $T$. This proves the claim.
	\vskip 5pt
	
	Now, if there exists an arrow outside the tree that doesn't point towards the tree, then there exists a vertex $v$ at distance at least $1$ of the tree that is a fork. Hence one of the face containing $v$ is in $\Upsilon^{(2)} (K)$ which contradicts the connectedness of $T_\psi (K)$.
\end{proof}

Using the function $H_{\psi}^{\ast}$, we have also the following characterization of the edges of $T_\psi$.

\begin{Lemma}\label{lem:edgephi}
	 Let $e \in \Upsilon^{(1)}$. Then $e \in T_\psi (K)$ if and only if there exists $\xi = (\a , \b) \in \Upsilon^{(2)} (K)$, and $\g \in \Upsilon^{(3)} (H_{\psi}^{\ast} (\xi , K))$, such that $e = \g \cap \xi$.
\end{Lemma}

\begin{proof}
	This is a direct consequence of the definition.
\end{proof}

This Lemma leads to the following property of Markoff maps in $({\bf \Psi}_{\Uom})_Q$.

\begin{Lemma}\label{lem:tree}
	Let $\psi \in ({\bf \Psi}_{\Uom})_Q$,  then the tree $T_\psi (2+M)$ is a finite attracting subtree.
\end{Lemma}

\begin{proof}
	Suppose $\psi \in ({\bf \Psi}_{\Uom})_Q$, then the set $\Upsilon_\psi^{(2)} (2+M)$ is finite, and for each element $\xi \in \Upsilon_\psi^{(2)} (2+M)$, we have $\psi (\xi) \notin [-2 , 2]$ and $\sigma (\xi) \neq 0$. Hence, the function $H_{\psi} (\xi, 2+M)$ is finite, which means that the subarc $J_{2+M} (\xi)$ is finite. So the subtree $T_\psi (2+M)$ is a finite union of finite subarcs. 
\end{proof}

\section{Fibonacci growth and proof of Theorem \ref{Main}} \label{s:fibonacci}

In this section we define the notion of Fibonacci growth for a function defined on $\Upsilon^{(3)}$, and prove that a map satisfying the (BQ)-conditions has Fibonacci growth. We will then use this to prove that the set $({\bf \Psi}_{\Uom})_Q$ of Markoff maps satisfying the (BQ)-conditions is an open domain of discontinuity for the mapping class group action, and we will also give different characterizations for the corresponding representations in $\X_Q(N)$.

\subsection{Fibonacci growth}

\subsubsection{Fibonacci functions}

We need to introduce the following notation. Given an oriented edge $\vec{e}$, we define the set $\Upsilon^{(k),0}(\vec{e})$, where $k = 2, 3$, as the subset of $\Upsilon^{(k)}$ given by the three $k$--simplices (either regions or faces) containing $\vec{e}$.  Removing $\vec{e}$ from $\Upsilon$ leaves two disjoint subtrees $\Upsilon^\pm$. Let $\Upsilon^+$ be the one containing the head of $\vec{e}$. We define $\Upsilon^{(k),-}(\vec{e})$ to be the subset of $\Upsilon^{(k)}$ given by the $k$--simplices whose boundary edges lie in $\Upsilon^-$, and similarly for $\Upsilon^{(k),+}(\vec{e})$. We have the following decomposition: 
$$\Upsilon^{(k)} = \Upsilon^{(k), -}(\vec{e}) \cup \Upsilon^{(k), 0}(\vec{e}) \cup \Upsilon^{(k), +}(\vec{e})$$
Note that $\Upsilon^{(k), +}(\vec{e}) = \Upsilon^{(k), -}(-\vec{e})$, where $-\vec{e}$ is the same edge as $\vec{e}$, pointing in the opposite direction.  Let $\Upsilon^{(k), 0 \pm}(\vec{e}) = \Upsilon^{(k), \pm}(\vec{e}) \cup \Upsilon^{(k), 0}(\vec{e}).$ 

We put the standard metric on the $1$--skeleton $\Upsilon^{(1)}$ graph $\Upsilon$, where every edge has length one. Given an edge $e$, we need to define the distance $d_e(X)$, where $X\in \Upsilon^{(3)}\cap \Upsilon^{(2)}$. If $X\in \Upsilon^{(k), 0-}(\vec{e})$, then we say that $d_e(X) = d_v(X)$, where $v$ is the head of $\vec{e}$, while, if $X\in \Upsilon^{(k), 0-}(-\vec{e})$, we say that $d_e(X) = d_v(X)$, where $v$ is the head of $-\vec{e}$.

There are two types of Fibonacci function that we will use: one defined over the set of regions $\Upsilon^{(3)}$, and one defined over the set of faces $\Upsilon^{(2)}$. This will allow us to give the notion of `Fibonacci growth'. In Proposition \ref{prop:fib_equivalence} we will see that the two notions are related, and that they express the word length of the elements that they represent.  

\begin{figure}
[hbt] \centering
\includegraphics[height=4.5 cm]{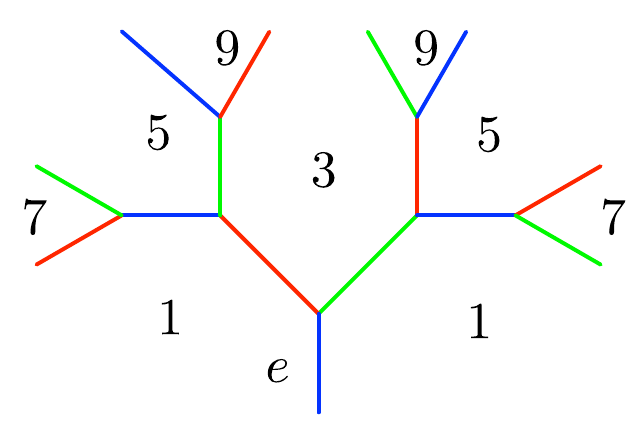}
\caption{The Fibonacci function $F_e\co \Upsilon^{(3)} \to \mathbb{Z}_{\geq 0}$. At each vertex there is an additional edge, not shown in the picture.}
\label{fig:Fib_edge}
\end{figure}

Fix an edge $e$ and define the function $$F_e \co \Upsilon^{(3)} \to \mathbb{Z}_{\geq 0}$$
as follows. We set  $F_e (\a) = 1$ on the three regions around $e$ (on $ \Upsilon^{(3), 0}(\vec{e})$ ), and for every `new' region $\alpha \in \Upsilon^{(3)}$ we set $F_e (\a)$ as the sum of the (already assigned) values of the other three regions meeting $\alpha$ at the same vertex. More formally, we define $F_e$ as follows:
$$F_e(\alpha) =
\left\{
	\begin{array}{ll}
		1  & \mbox{if } d_e(\a) = 0, \\
		F_e(\beta)+F_e(\gamma)+F_e(\delta) & \mbox{if } (\alpha, \beta, \gamma, \delta) \in \Upsilon^{(0)} \mbox{ and } d_(\beta), d_e(\gamma), d_e(\delta) < d_e(\alpha).
	\end{array}
\right.$$
This notion is due to Bowditch \cite{bow_mar}. In Figure \ref{fig:Fib_edge}, you can see the value of $F_e$ on the region given by the edges in the picture, and an additional edge perpendicular to the face (and coming out of it).

\begin{figure}
[hbt] \centering
\includegraphics[height=4.5 cm]{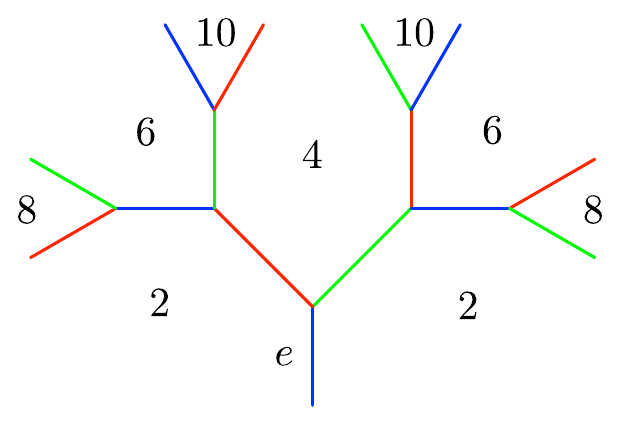}
\caption{The Fibonacci function $F_e\co \Upsilon^{(2)} \to \mathbb{Z}_{\geq 0}$. (Note that at every vertex there is a forth edge not drawn.)}
\label{fig:fib_vertex}
\end{figure}

If we consider the faces, we define $$F_e \co \Upsilon^{(2)} \to \mathbb{Z}_{\geq 0}$$ to be $2$ on the three faces around $e$, and then we define $F_e$ on the remaining faces by assigning to every `new' face $(\alpha, \beta) \in \Upsilon^{(2)}$ the sum of the (already assigned) values of the other two faces meeting $(\alpha, \beta)$ at the same vertex, and having shorter distance from $e$. More formally, we define $F_v$ as follows:
  $$F_e(\xi) =
  \left\{
  	\begin{array}{ll}
  		2  & \mbox{if } d_e(\xi) = 0 \\
  		F_e(\xi')+F_e(\xi'') & \mbox{if } d_e(\xi'), d_e(\xi'')< d_e(\xi), \xi \cap \xi' \in \Upsilon^{(1)},  \xi \cap \xi'' \in \Upsilon^{(1)}.
  	\end{array}
  \right.$$
This notion is a modification of the notion of Fibonacci function given by Hu-Tan-Zhang \cite{hu_pol}, who define the Fibonacci function $F_v$ with respect to a vertex $v$. Recall also the correspondence between faces and bi-colored geodesics, underlined in Remark \ref{bi-colored}. In Figure \ref{fig:fib_vertex} we only draw the faces in the boundary of a region, but in Figure \ref{fig:fib_vertex2} you can see a more complete picture.

\begin{figure}
[hbt] \centering
\includegraphics[height=5.5 cm]{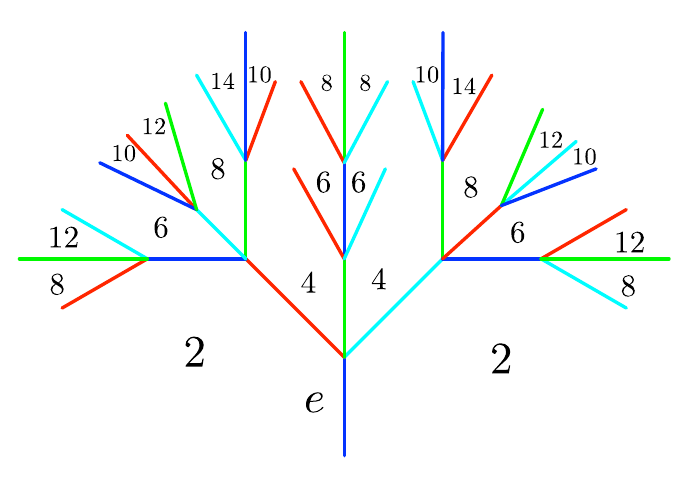}
\caption{The Fibonacci function $F_e$ on (some of) the faces (or bi-colored geodesics).}
\label{fig:fib_vertex2}
\end{figure}

The following lemma can be easily proved by induction. Its corollary shows that the concept of upper and lower Fibonacci bound is independent of the edge $e$ used.

\begin{Lemma}\label{fib_bound}{\rm (Lemma 2.1.1 \cite{bow_mar}; Lemma 25-26-27 \cite{hua_sim})}
Let $e$ be an edge that is the intersection of the three regions $\a_1, \a_2, \a_3$ and let $f\co \Upsilon^{(3)} \rightarrow [0, \infty)$ be a function. Let $M = \mathrm{max}\{f(\a_1), f(\a_2), f(\a_3)\}$,  $m = \mathrm{min}\{f(\a_1), f(\a_2), f(\a_3)\}$.   Let $\a, \b, \g, \d$ be four regions that meet at the same vertex and such that $\d$ is strictly farther from $e$ than $\a, \b, \g$. Then:
\begin{enumerate}
  \item  $f(\d)\leq f(\a)+f(\b)+f(\g)+2c,$ for all such $\a, \b, \g, \d$ then we have that $f(X)\leq (M+c)F_e(X)-c,$ for all but finitely many $X \in \Upsilon^{(3)}$.
  \item If there exists $0 \leq c < m $ such that $f(\d)\geq f(\a)+f(\b)+f(\g)-2c,$ for all such $\a, \b, \g, \d$, then we have that $f(X)\geq (m-c)F_e(X)+c,$ for all but finitely many $X \in \Upsilon^{(3)}$.
  \end{enumerate}
\end{Lemma}

\begin{Corollary}\label{fib_cor_upper}{\rm (Corollary 2.1.2 in \cite{bow_mar})}
Suppose $f\co \Upsilon^{(3)} \rightarrow [0, \infty)$ satisfies an inequality of the form $f(\d) \le f(\a)+f(\b)+f(\g)+c$ [resp. $f(\d) \ge f(\a)+f(\b)+f(\g)+c$] for some fixed constant $c$, whenever $\a, \b, \g, \d \in \Upsilon^{(3)}$ meet at a vertex. Then for any given edge $e \in \Upsilon^{(1)}$, there is a constant $K>0$, such that $f(X) \le KF_e(X)$ [resp. $f(X) \ge KF_e(X)$] for all $X \in \Upsilon^{(3)}$.
\end{Corollary}

\subsubsection{Fibonacci growth}

We can now define when a function $f\co\Upsilon^{(3)} \to \R$, or $g\co\Upsilon^{(2)} \to \R$, has an upper or lower Fibonacci bound, or Fibonacci growth. Notice that this definition say something about the asymptotic growth of the function, and hence it is independent of the edge $e$ used in the definition of the respective Fibonacci function.

\begin{Definition}
  Suppose $f\co\Upsilon^{\prime}\subset\Upsilon^{(3)} \to \R$ or $g\co\Upsilon^{\prime}\subset\Upsilon^{(2)} \to \R$. We say that:
  \begin{itemize}
    \item $f$ [resp. $g$] has an {\em upper Fibonacci bound on} $\Upsilon^{\prime}$ if there is some constant $\kappa > 0$ such that $f(X) \le \kappa\, F_e(X)$ [resp. $g(X) \le \kappa\, F_e(X)$] for all but finitely many $X \in \Upsilon^{\prime}$ ;
    \item $f$ [resp. $g$] has a {\em lower Fibonacci bound on} $\Upsilon^{\prime}$ if there is some constant $\kappa > 0$ such that $f(X) \ge \kappa\, F_e(X)$ [resp. $g(X) \ge \kappa\, F_e(X)$] for all but finitely many $X \in \Upsilon^{\prime}$;
    \item $f$ [resp. $g$] has {\em Fibonacci growth} on $\Upsilon^{\prime}$ if it has both upper and lower Fibonacci bounds on $\Upsilon^{\prime}$;
    \item $f$ [resp. $g$] has {\em Fibonacci growth} if it has Fibonacci growth on {\it all} of $\Upsilon^{(3)}$ [resp. $\Upsilon^{(2)}$].
  \end{itemize}
\end{Definition}

\begin{Remark}
  Note that, equivalently, we can say that:
  \begin{itemize}
    \item $f$ [resp. $g$] has an {\em upper Fibonacci bound on} $\Upsilon^{\prime}$ if there are some constants $\kappa, C > 0$ such that $f(X) \le \kappa\, F_e(X) +C$ [resp. $g(X) \le \kappa\, F_e(X) +C$] for all $X \in \Upsilon^{\prime}$ ;
    \item $f$ [resp. $g$] has an {\em lower Fibonacci bound on} $\Upsilon^{\prime}$ if there are some constants $\kappa, C > 0$ such that $f(X) \ge \kappa\, F_e(X)-C$ [resp. $g(X) \ge \kappa\, F_e(X) -C$] for all $X \in \Upsilon^{\prime}$.
  \end{itemize}
\end{Remark}

The function that we will consider is the function $\log^{+}|\psi| \co \Upsilon^{(2)}\cup \Upsilon^{(3)}\to \R$, where $\psi \in {\bf \Psi}_{\Uom}$ and  $\log^{+}|\psi|:= \max\{\log|\psi|, 0\}$. For any $\Uom$-Markoff map $\psi$ the function $\log^{+}|\psi|$ will have an upper Fibonacci bound, but we will need to restrict to maps in the Bowditch domain to get maps having a lower Fibonacci bound (and so a Fibonacci growth).

In particular, we will prove that, when $\psi\in ({\bf \Psi}_{\Uom})_Q$, then $\log^{+}|\psi|$ has Fibonacci growth on $\Upsilon^{(3)}$. The following result will then show that $\log^{+}|\psi|$ has Fibonacci growth on $\Upsilon^{(2)}$ as well.

\begin{Proposition}\label{prop:fib_equivalence}
Let $\psi\in ({\bf \Psi}_{\Uom})_Q$. The function $\log^{+}|\psi|$ has Fibonacci growth on $\Upsilon^{(3)}$ if and only if $\log^{+}|\psi|$ has Fibonacci growth on $\Upsilon^{(2)}$.
\end{Proposition}

\begin{proof}
  First, note that given a face $\xi = (\a, \b) \in \Upsilon^{(2)}$, we have that  $$|\Psi(\a)| \cdot |\Psi(\b)|-M \leq |\Psi(\xi)| \leq |\Psi(\a)| \cdot |\Psi(\b)|+M.$$ 
  
  Since $\psi \in ({\bf \Psi}_{\Uom})_{Q}$, we know that the set $\Upsilon^{(2)}_\psi(2+M)$ is finite, and so there is just a finite number of faces such that $|\psi(e)| \leq (M+2)^2+M$. We can then suppose $|ab-x| > (M+2)^2+M$, and so $|ab|> (M+2)^2$. In that case we have 
$$\frac{M+1}{M+2}|ab| \leq |ab-x| \leq \frac{M+3}{M+2}|ab|,$$ because 
$$|ab|-M = \frac{M+1}{M+2}|ab| + \frac{1}{M+2}|ab|- M > \frac{M+1}{M+2}|ab| + (M+2) - M > \frac{M+1}{M+2} |ab|,$$ and similarly
$$|ab|+M = \frac{M+3}{M+2}|ab| - \frac{1}{M+2}|ab|- M < \frac{M+3}{M+2}|ab|- (M+2) + M < \frac{M+3}{M+2} |ab|.$$  

Hence we can see that $$\log^{+}|\Psi(\a)| + \log^{+}|\Psi(\b)|-K \leq \log^{+}|\Psi(\xi)| \leq \log^{+}|\Psi(\a)| + \log^{+}|\Psi(\b)|+K',$$ and so, using Proposition \ref{prop:word}, we can conclude.

Conversely, we can see that 
\begin{equation}\label{eq1}
 \frac{M+2}{M+3}|ab-x| \leq |ab| \leq \frac{M+2}{M+1}|ab-x|. 
\end{equation}
Similarly, we have 
 \begin{equation}\label{eq2}
  \frac{M+2}{M+3}|ab-z| \leq |ac| \leq \frac{M+2}{M+1}|ac-z|, 
 \end{equation}
 and
 \begin{equation}\label{eq3}
  \frac{M+2}{M+3}|bc-y| \leq |bc| \leq \frac{M+2}{M+1}|bc-y|. 
 \end{equation}
So, using Proposition \ref{prop:word}, and adding Equations \eqref{eq1} and \eqref{eq2}, and subtracting \eqref{eq3}, we can conclude.
 \end{proof}

Another important meaning of these Fibonacci functions is explained by the following result, relating $F_e$ to the word length of the elements of $\Sc \subset \Gamma/\!\sim$. (See  Section \ref{ss:simple} for the definition of $\Gamma/\!\sim$.) In fact, we can write out explicit representatives in $\Gamma$ of elements of $\Gamma/\!\sim$ corresponding to given elements of $\Sc$, as follows.

Let $e = (\a, \b, \g; \d, \d')$. The regions $\a, \b, \g \in \Sc_1$ are represented, respectively, by a triple of free generators for $\Gamma$. Without loss of generality, we can suppose that $\d$ and $\d'$ are represented by $\a\b\g$ and $\a\g\b$, respectively, and that the faces $\xi_{\a, \b} = \a \cap \b \in \Sc_2$ are represented by $\a\b^{-1}$. We can now inductively give representatives for all other elements of $\Upsilon^{(2)}\cup \Upsilon^{(3)}$. Note that all the words arising in this way are cyclically reduced. Let $\mathrm{W}(\omega)$ denote the minimal cyclically reduced word length of element $\omega$ with respect to some generating set. Since it doesn't matter which generating set we choose, we may as well take it to be a free basis.  From this we deduce the following:

\begin{Proposition}\label{prop:word}
Suppose $a$, $b$ and $c$ are a set of free generators for $\Gamma$ corresponding to regions $\a_1$, $\a_2$ and $\a_3$. Let $e$ be the edge $e = \a_1 \cap \a_2 \cap \a_3$. If $\omega_X \in \hat\Omega$ correspond to $X\in \Upsilon^{(k)}$, where $k = 2$ or $k =3$, then $W(\omega_X) = F_e(X)$. 

In addition, for any face $\xi_{\a, \b} = \a \cap \beta$, where $\a, \b \in \Upsilon^{(3)}$, we have $F_e(\xi_{\a, \b}) = F_e(\a) + F_e(\b).$
\end{Proposition}

\subsubsection{Upper Fibonacci bound} 

In this section we will show that, for any $\Uom$-representation $\psi$, the function $\log^{+}|\psi|$ has an upper Fibonacci bound on $\Upsilon^{(3)}$ (and so on $\Upsilon^{(2)}$, by Proposition \ref{prop:fib_equivalence}).

\begin{Lemma}\label{lem:upper_bound}
If $\psi \in {\bf \Psi}_{\Uom}$, then $\log^{+}|\psi|$ has an upper Fibonacci bound on $\Upsilon^{(3)}$.
\end{Lemma}

\begin{proof}
We will follow the ideas of \cite{tan_gen}. Let $\a, \b, \g, \d \in \Upsilon^{(3)}$ be four regions that meet at a vertex, and let $a = \psi(\a)$, $b = \psi(\b)$, etc. We will prove that
\begin{equation}\label{eqn:upper_bound}
  \log^{+}|d| \le \log^{+}|a|  + \log^{+}|b| + \log^{+}|c| + \log 32,
\end{equation}
and then conclude using Corollary \ref{fib_cor_upper}.

If $|d| \leq 2 \max \{ |a|, |b|, |c| \}$, then \eqref{eqn:upper_bound} holds already. So we suppose $|d| > 2 \max \{ |a|, |b|, |c| \}$. Then, since $-abcd + x(ab+cd) + y (bc+ad) + z(ac+bd) +4 - x^2 -y^2 - z^2 -xyz = a^2 + b^2 + c^2 + d^2$, we have:
  \begin{align*}
  &|abcd|+ |x| |ab| + |x| |cd| + |y||bc| + |y| |ad| + |z| |ac| + |z||bd|+|4 - x^2 -y^2 - z^2 -xyz| \\
  &\ge |d|^2-|a|^2-|b|^2 -|c|^2\\
  &= |d|^2/4 + (|d|^2/4-|a|^2) + (|d|^2/4-|b|^2) + (|d|^2/4-|c|^2) \\
  &\ge |d|^2/4.
  \end{align*}

Hence $|d|^2 \le 32 m$, where $$m = \max\{|abcd|, |x| |ab|, |x| |cd|, |y||bc|, |y| |ad|, |z| |ac|, |z||bd|, |4 - x^2 -y^2 - z^2 -xyz|\}.$$ So, according to the value of $m$, we have, respectively:
\begin{enumerate}
  \item $|d|^2 \le 32 |abcd|$, so $|d| \le 32 |abc|$;
  \item $|d|^2 \le 32 |x||ab|$, so $|d| \le |d|^2 \le 32 |x||ab|$;
  \item $|d|^2 \le 32 |x||cd|$, so $|d| \le 32 |x||c|$;
  \item $|d|^2 \le 32 |4 - x^2 -y^2 - z^2 -xyz|$, so $|d| \le |d|^2 \le 32 |4 - x^2 -y^2 - z^2 -xyz|$;
\end{enumerate}
or a similar inequality.
Note that we have $|d| \le |d|^2$, since we may assume $|d| \ge 1$.

From this, Equation \eqref{eqn:upper_bound} follows easily.

\end{proof}

\subsubsection{Lower Fibonacci bound} 

In this section we will show that, for any $\Uom$-representation $\psi$ in the Bowditch set, the function $\log^{+}|\psi|$ will have a lower Fibonacci bound (and so Fibonacci growth) on $\Upsilon^{(3)}$. We will then use that to prove that $\log^{+}|\psi|$ has Fibonacci growth on $\Upsilon^{(2)}$.

\begin{Theorem}\label{thm:lower_bound}
If $\psi \in ({\bf \Psi}_{\Uom})_{Q}$, then $\log^{+}|\psi|$ has lower Fibonacci growth on $\Upsilon^{(3)}$.
\end{Theorem}

\begin{proof}
From Lemma \ref{lem:tree}, we know that, if $\psi \in ({\bf \Psi}_{\Uom})_{Q}$, then there is a finite $\psi$--attracting subtree $T_\psi = T_\psi(2+M)$, where $M = \mathrm{max}\{|x|, |y|, |z|\}$.  Remember that $C(T_\psi)$ is the circular set of directed edges given by $T_\psi$. Note that $\Upsilon_\psi^{(3)}(2+M) = \cup_{\vec{e} \in C(T_\psi)}\Upsilon^{(3),0-}(\vec{e})$. So it suffices to prove that for each edge $\vec{e}$ in $C(T_\psi)$, $\log^{+}|\psi|$ has lower Fibonacci growth on $\Upsilon^{(3), 0-}(\vec{e})$. 

\vskip 3pt
\underline{\textit{Case 1}:} $\Upsilon^{(3), 0-}(\vec{e}) \cap \Upsilon_\psi(2+M) = \emptyset$.

Let $c = \mathrm{min}\{\frac{\psi(\alpha)}{2} | \alpha \in \Upsilon^{(3), 0}(\vec{e}) \}$. Then $c>0$, and for $\a\in \Upsilon^{(3), 0}(\vec{e})$, we have $\frac{\psi(\alpha)}{2}>c = c F_e(\alpha)$.

Now let $\alpha \in \Upsilon^{-}(\vec{e})$, and let $\b, \g, \d, \a' \in \Upsilon^{(3)}$ having shorter distance from $e$, and such that $(\a, \b, \g, \d)\in \Upsilon^{(0)}$ and $(\a', \b, \g, \d)\in \Upsilon^{(0)}$. We will prove that
\begin{equation}\label{eqn:upper}
  \log^{+}|a| \ge \log^{+}|b|  + \log^{+}|c| + \log^{+}|d| - c,
\end{equation}
and then conclude using Lemma \ref{fib_bound} (ii).

If $|a| \geq 2 \min \{  |b|, |c|, |d| \}$, then \eqref{eqn:upper} holds already. So we suppose $|a| > 2 \max \{ |b|, |c|, |d| \}$ Then, using the edge relation, and the fact that $|a'| \le |a|$, we have that $$|bcd|\le |x||b| + |z||c| + |y||d|+ |a| + |a'|\leq M(|b| + |c| + |d|) + 2|a|\le (3M+2)|a|,$$ so we have $$\log^{+}|a| \ge \log^{+}|b|  + \log^{+}|c| + \log^{+}|d| - \log(3M+2).$$

\vskip 3pt

\underline{\textit{Case 2}:} $\Upsilon^{(3), 0-}(\vec{e}) \cap \Upsilon_\psi(2+M) \neq \emptyset$.

First notice that $\Upsilon^{-}(\vec{e}) \cap \Upsilon_\psi(2+M) = \emptyset$, and $\Upsilon^{(3), 0}(\vec{e}) \cap \Upsilon_\psi(2+M) = \{\a\}$. The reason is that $\xi \in C(T_\psi)$. 

\begin{figure}
[hbt] \centering
\includegraphics[height=5.5 cm]{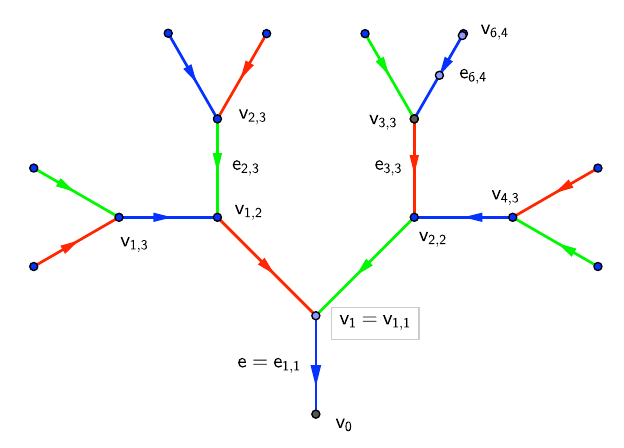}
\caption{The labelling of the vertices of $\a$ intersecting $\Upsilon^{-}$.}
\label{fig:region}
\end{figure}

Now, let $v=v_0$ and $v_1 = v_{1,0}$ be the head and the tail of $e$, and let $v_{i, k}$ where $i = 1, \ldots, 2^{k-1}$, and $k\geq 0$ be a (choice of a) labelling of the vertices of $\a$ intersecting $\Upsilon^{-}$, so that the distance $d(v_{i, k}, e) = k$. Let $e_{i, k}$ be the edge of $\a$ intersecting $\Upsilon^{-}$ and with tail $v_{i, k}$, and let $\eta_{i, k}$ be the other (oriented) edge incident at $v_{i,k}$ and which don't belong to $\a$. See Figure \ref{fig:region}. %

Note that $\Upsilon^{(3), 0-}(\vec{e}) = \{\alpha\} \cup (\bigcup_{k \geq 1} \bigcup_{i = 1}^{2^k} \Upsilon^{(3), 0-}(\eta_{i, k}))$. In addition, by Lemma \ref{lem:neighbors}, the regions in $\Upsilon^{(3), 0-}(\eta_{i, k})$ have exponential growth, and so we can conclude by Lemma \ref{fib_bound} (ii).
\end{proof}

Using Lemma \ref{lem:upper_bound}, Theorem \ref{thm:lower_bound} and Proposition \ref{prop:fib_equivalence}, we have now the following result.

\begin{Corollary}\label{cor:fib_growth}
If $\psi \in ({\bf \Psi}_{\Uom})_{Q}$, then $\log^{+}|\psi|$ has Fibonacci growth on $\Upsilon^{(2)}$ and $\Upsilon^{(3)}$.
\end{Corollary}

\subsection{Characterization of BQ-set} \label{s:characterization}

We can now state two different characterizations for the representations in the Bowditch set, and we will use the result from the previous section to prove that the two definitions are equivalent.

First, let's see the equivalence with the definition of the Bowditch set given in the Introduction.

\begin{Proposition}\label{pro:equiv_def}
  A map $\psi \in ({\bf \Psi}_{\Uom})_{Q}$ if and only if the two conditions are satisfied:
  \begin{enumerate}
  	\item[(BQ1)] For all $\xi \in \Upsilon^{(2)}$, we have $\psi(\xi) \notin [-2 , 2 ]$.
  	\item[(BQ2)] For all $K>0$, the set $\Upsilon^{(2)} (K)$ is finite.
  \end{enumerate}
\end{Proposition}

\begin{proof}
	If $\psi \in ({\bf \Psi}_{\Uom})_{Q}$ then, by Corollary \ref{cor:fib_growth}, the map $\log^{+} |\psi|$ has Fibonacci growth on $\Upsilon^{(2)}$, which means that for all $K > 0$, the set $\Upsilon^{(2)} (K)$ is finite.
	
	Reciprocally, if $\psi$ is not in $({\bf \Psi}_{\Uom})_{Q}$, then at least one of the three possibilities occurs:
	\begin{enumerate}[(i)]
		\item There exists $\xi \in \Upsilon^{(2)}$, such that $\psi (\xi) \in [-2 , 2 ]$.
		\item The set $\Upsilon^{(2)} (2+M)$ is infinite. 
		\item There exists $\xi \in \Upsilon^{(2)}$, such that $\sigma (\xi) = 0$. 
	\end{enumerate}
	
	Then $(i)$ contradicts condition $(BQ1)$, while $(ii)$ contradicts condition $(BQ2)$. Finally, in case $(iii)$, the neighbors around the face $\xi$ will converge to some finite value. Which means that there exists $K>0$ such that $\Upsilon^{(2)} (K)$ is infinite, contradicting the condition $(BQ2)$.
\end{proof}

\begin{Remark}\label{rk:sigma}
	Let $\rho$ be a representation  and $\psi$ the corresponding Markoff map. There is an equivalence between the conditions 
	\begin{itemize}
		\item There exists $\xi \in \Upsilon^{(2)}$, such that $\sigma_\psi (\xi) = 0$. 
		\item There exists an embedded subsurface $S \subset N_{1,3}$ such that $\chi(S) = -1$ and the restriction $\rho|_{\pi_1 (S)} $ is reductive. 
	\end{itemize}
\end{Remark}

\begin{proof}
  If $\a$ and $\b$ are two curves such that $\xi = \xi_{\a, \b}$, then the curve $\a \b^{-1}$ splits $N$ into a three-holed sphere  $S'$ and a two holed projective plane $N'$. The fundamental group $\pi_1 (N')$ is generated by $\{ \a , \b \}$ and hence the restriction of $\rho$ to $\pi_1 (N')$ is reductive if and only if 
	$$\Tr ([\a , \b ]) = a^2 + b^2 + x^2 - abx -2 = 2.$$
	
	Similarly, the fundamental group $\pi_1 (S')$ is generated by two elements $Y$ and $Z$ which are conjugated to the boundaries of $N$ and such that $YZ = \a \b^{-1}$. The restriction of $\rho$ to $\pi_1 (S')$ is reductive if and only if 
	$$\Tr ([ Y , Z ]) = y^2 + z^2 + \psi (\xi)^2 - yz \psi(\xi) - 2 = 2.$$
	As $\sigma_{\psi} (\xi) = (\Tr( [\a , \b])-2)(\Tr ([Y , Z ]) - 2)$, we get the equivalence directly.
\end{proof}

The following second correspondence is particularly interesting because it could be used to generalize the notion of Bowditch representations for general surfaces. Given $\rho \in \mathfrak{X}$, we can define a function $L_\rho = L(\rho(\cdot)) \co \Sc \to \C$, where $\Sc = \Sc(N)$ is the set of simple closed curve of $N = N_{1,3}$ by $$\Tr(\rho(\g)) = 2 \mathrm{cosh}(\mathrm{L}(\rho(\g))/2).$$
Note that $\log|\Tr(\rho(\g))| \leq \mathrm{W}(\rho(\g))$, see for example Bowditch \cite{bow_mar}. Recall that $\mathrm{W}$ denote the minimal cyclically reduced word length with respect to some generating set. 

\begin{Proposition}\label{pr:growth}
  The following sets are equal:
  \begin{enumerate}
    \item the Bowditch set $(\mathfrak{X}_{\Uom})_Q$;
    \item $\{\rho \in \mathfrak{X}_{\Uom} \mid \exists k = k(\rho), \; |\mathrm{L}(\rho(\g))| \geq k \mathrm{W}(\g)\;\;\forall \g \in \Sc \};$
    \item $\{\rho \in \mathfrak{X}_{\Uom} \mid \exists k = k(\rho), \; |\mathrm{L}(\rho(\g))| \geq k \mathrm{W}(\g) \;\;\forall \g \in \Sc_1 \};$
    \item $\{\rho \in \mathfrak{X}_{\Uom} \mid \exists k = k(\rho), \; |\mathrm{L}(\rho(\g))| \geq k \mathrm{W}(\g) \;\;\forall \g \in \Sc_2 \}.$
  \end{enumerate}
\end{Proposition}

\begin{proof}
Let $\mathfrak{S}, \mathfrak{S}_1, \mathfrak{S}_2$ the sets defined in $(2), (3), (4)$ of Proposition \ref{pr:growth}, respectively. Note that $\mathfrak{S} \subset \mathfrak{S}_1$ and $\mathfrak{S} \subset \mathfrak{S}_2$.  
  
Given $\rho \in (\mathfrak{X}_{\Uom})_Q$, let $\psi \in ({\bf \Psi}_{\Uom})_{Q}$ be the corresponding ${\Uom}$--Markoff map. Proposition \ref{prop:word} tells us that, for the appropriate edge $e$, we have $W(\gamma) = F_e(\g)$, if $\g \in \Upsilon^{(2)}\cup \Upsilon^{(3)}$. By Corollaries \ref{cor:fib_growth} there is a constant $k > 0$ such that $\log^+|\psi(X)| \geq k F_e(X)$ for all $X \in \Upsilon^{(2)}\cup \Upsilon^{(3)}$. Recall that $\psi(X) = \Tr(X)$, so, from the inequality $\log^+|\Tr(A)| \leq \mathrm{L}(A)$, we can see that $\mathrm{L}(\rho(\g)) \geq \log^+|\psi(X)| \geq k F_e(X) = k W(\g)$. This proves that $(\mathfrak{X}_{\Uom})_Q \subset \mathfrak{S}$, and hence $(\mathfrak{X}_{\Uom})_Q \subset \mathfrak{S}_1$ and $(\mathfrak{X}_{\Uom})_Q \subset \mathfrak{S}_2$.

  Conversely, given $\rho \in \mathfrak{S}_2$, we can prove that:
  \begin{itemize}
    \item[(BQ1)] For all $\g \in \Sc_2$, we have $\Tr(\rho(\g)) \notin [-2,2]$;
    \item[(BQ2)] For all $K>0$, the set $\{\g \in \Sc_2 \mid |\Tr(\rho(\g))| < K\}$ is finite.
  \end{itemize}
 In fact, if there exist $\g \in \Sc_2$ such that $\Tr(\rho(\g)) \in [-2, 2]$, then $\rho \notin \mathfrak{S}_2$, because we can find a sequence of elements in $\Sc_2$ such that the word length increases, but the length remains bounded.
 On the other hand, if there exist $K >0$ such that $\{\g \in \Sc_2 \mid |\Tr(\rho(\g))| < K\}$ is infinite, then, again $\rho \notin \mathfrak{S}_2$, because for any $K >0$ there is only a finite number of elements with word length less or equal $K$. This proves $\mathfrak{S}_2 \subset (\mathfrak{X}_{\Uom})_Q$, so $(\mathfrak{X}_{\Uom})_Q = \mathfrak{S} = \mathfrak{S}_2$. Now using Proposition \ref{prop:fib_equivalence}, we can also conclude that $\mathfrak{S} = \mathfrak{S}_1 = \mathfrak{S}_2 = (\mathfrak{X}_{\Uom})_Q$. 
\end{proof}

An easy corollary of this result is the inclusion $\X_{\mathrm{ps}} \subset (\mathfrak{X}_{\Uom})_Q$. Indeed, the set $\X_{\mathrm{ps}}$ corresponds to representations such that the axes corresponding to primitive elements are uniform quasi-geodesics, that is,
$$ \X_{\mathrm{ps}} = \{\rho \in \X \mid \exists k = k(\rho), \mathrm{L}(\rho(\g_1^{-1}\g_2))| \geq k \mathrm{W}(\g_1^{-1}\g_2)\;\;\forall \g_1, \g_2 \in \mathrm{Ax}(\g), \g \; \text{primitive}\},$$
where $\mathrm{Ax}(\g)$ is the axis passing through the identity $e$. Now, taking $\g_1 = e$, $\g_2 = \g$ and noticing that all elements in $\Sc$ are primitive, we can prove the above inclusion. To see that this inclusion is proper, we can note that the holonomy of an hyperbolic structure with punctures at the three boundary components gives a representation in $(\mathfrak{X}_{\Uom})_Q$, but not in $\X_{\mathrm{ps}}$. So we have the following:

\begin{Proposition}\label{pr:primitivestable}
  $\X_{\mathrm{ps}} \subsetneq (\mathfrak{X}_{\Uom})_Q.$
\end{Proposition}

Finally, we prove this alternative characterization of Bowditch maps in terms of the attracting subtree:

\begin{Proposition}\label{pro:finite_tree}
	Let $K\geq 2+M$. Then $\psi \in ({\bf \Psi}_{\Uom})_Q$ if and only if the tree $T_\psi (K)$ is finite.
\end{Proposition}

\begin{proof}
	Let $K \geq 2+M$. Suppose $\psi \in ({\bf \Psi}_{\Uom})_Q$, then the set $\Upsilon_\psi^{(2)} (K)$ is finite, and for each element $\xi \in \Upsilon_\psi^{(2)} (K)$, we have $\psi (\xi) \notin [-2 , 2]$ and $\sigma (\xi) \neq 0$. Hence, the function $H_{\psi} (\xi)$ is finite, which means that the subarc $J_\psi (\xi, K)$ is finite for all $K \geq 2+M$. The tree $T_\psi (K)$ is now a finite union of finite subarcs. 
	
	Now suppose $\psi \notin ({\bf \Psi}_{\Uom})_Q$, then we have three cases:
	\begin{enumerate}
		\item There exists $\xi$ such that $\psi(\xi) \in [-2 , 2 ]$. In this case $\xi \in \Upsilon_\psi^{(2)} (K)$, and the arc $J_\psi (\xi, K) = \xi$ is infinite.  So $T_\psi (K)$ is infinite. 
		\item The set $\Upsilon_\psi^{(2)} (K)$ is infinite. As for each $\xi \in \Upsilon_\psi^{(2)} (K)$, the arc $J_\psi (\xi, K)$ contains at least one edge, then the tree $T_\psi (K)$ is infinite.
		\item There exists $\xi$ such that $\sigma (\xi) = 0$. In this case, either $e \in \Upsilon_\psi^{(2)} (K)$, in which case the arc $J_\psi (\xi, K) = \xi$ is contained in $T_\psi (K)$, or $\Upsilon_\psi^{(2)} (K)$ is infinite. 
	\end{enumerate}
	In either cases, the tree $T_\psi (K)$ is infinite, which concludes the proof of the Proposition.
	\end{proof}

\subsection{Openness and proper discontinuity} 

First we prove the following lemma which will imply the openness of $({\bf \Psi}_{\Uom})_Q$.

\begin{Lemma}\label{lem:treestable}
Let $K > 2+M$. For each $\psi \in ({\bf \Psi}_{\Uom})_Q$,  if $T_\psi(K)$ is non-empty, there exists a neighborhood $U_\psi$ of $\psi \in {\bf \Psi}_{\Uom}$ such that $\forall \phi \in U_\psi$, the tree $T_\phi (K)$ is contained in $T_\psi (K)$.
\end{Lemma}

\begin{proof}
	Let $\psi \in ({\bf \Psi}_{\Uom})_Q$, and $K> 2+M$ such that $T_\psi (K)$ is non-empty. First, it is easy to see that for $\phi$ close enough to $\psi$, the trees $T_\psi (K)$ and $T_\phi (K)$ have non-empty intersection.
	
	 Now, consider the set $E$ of edges which meet $T_\psi (K)$ in a single point. Let $e \in E$. We can show that, if $\phi \in {\bf \Psi}_{\Uom}$ is close enough to $\psi$, then $e \notin T_\phi (K)$. Indeed, from Lemma \ref{lem:edgephi}, we have that:
	 $$e \notin T_\phi (K) \, \Leftrightarrow \forall (\a , \b , \g ) = e ,  \left\{ \begin{array}{rl}
	 			 		& |\phi(\a)|> K \mbox{ and } |\phi(\b)| > K \\
				\mbox{  or  } &	|\phi (\xi_{\a , \b}) | > K^2 + M \\
				\mbox{  or  } & |\phi (\g) | > H_{\psi}^{\ast} (\xi_{\a , \b} , K) \end{array} \right. $$
On the other hand, since the edge $e$ is not in $T_\psi (K)$, we know that one of those strict inequalities holds for $\psi$. So if we choose $\phi$ close enough, the corresponding inequality in $\phi$ will also hold and hence $e \notin T_\phi (K)$.
	 
Since the tree $T_\psi (K)$ is finite, there is only a finite number of edges in $E$. So again, we can choose $\phi$ close enough so that any edge $e \in E$ is not in $T_\phi (K)$.
	
	Now the tree $T_\phi (K)$ is connected and $T_\psi(K) \cap T_\phi (K)$ is non empty, hence $T_\phi (K)$ is entirely contained in $T_\phi (K)$.
\end{proof}

\begin{Theorem}
	The set $({\bf \Psi}_{\Uom})_Q$ is open in ${\bf \Psi}_{\Uom}$ and the action of $\Gamma$ is properly discontinuous.
\end{Theorem}

\begin{proof}
	The previous lemma directly implies openness. Indeed, let $\psi \in ({\bf \Psi}_{\Uom})_Q$ and $K > 2+M$ such that $T_\psi (K)$ is non empty. Then for each $\phi$ in the open neighborhood $U_\psi$ constructed in previous lemma, the tree $T_\phi (K)$ is contained in $T_\psi (K)$ and hence is finite. Which proves that $U_\psi \subset  ({\bf \Psi}_{\Uom})_Q$.
	
	To prove that the action is properly discontinuous we take $C$ be a compact subset in $({\bf \Psi}_{\Uom})_Q$. We want to prove that the set:
	$$\Gamma_0 = \{ g \in \Gamma \, \mid \, gC \cap C \neq \emptyset \}$$
	is finite. 
	
	Let $K>2+M$ such that any for any $\psi \in C$, the tree $T_\psi (K)$ is non-empty. Now around each element of $C$, there exists a neighborhood $U_\psi$ given by the Lemma \ref{lem:treestable}. So the set $(U_\psi)_{\psi \in C}$ is a open cover of $C$. We take a finite subcover $(U_{\psi_i})_{i \in I}$ where $I$ is a finite set. 
	
	Now for each element $\psi_i$ we take the tree $T_{\psi_i} (K)$, and consider the union
	$$\mathcal{T} = \bigcup_{i \in I} T_{\psi_i} (K).$$
	
	By construction, for each element $\phi \in C$, the tree $T_\phi (K)$ is contained in $\mathcal{T} $. The tree $\mathcal{T} $ is a finite union of finite trees and hence is itself finite. It follows that the set:
	$$\Gamma_1 = \{ g \in \Gamma \, \mid \, g \mathcal{T} \cap \mathcal{T} \neq \emptyset \} \mbox{    is finite .}$$
	As $T_{g \phi} (K) = g T_{\phi} (K)$, then $\Gamma_0 \subset \Gamma_1$, and hence $\Gamma_0$ is finite. This proves that the action is properly discontinuous.
\end{proof}

\section{Concluding Remarks} 

\subsection{Generalization to other surfaces}

In this paper we are discussing the existence of a domain of discontinuity for the action of the mapping class group $\MCG(N_{1,3})$ of the three holed projective plane on $\X$. A natural idea is to generalize this theory to the case of a general orientable surface $\Sigma_{g,b}$ or non-orientable ones $N_{g,b}$. It turns out that the different characterizations of the Bowditch set $\X_Q$ given in Section \ref{s:characterization} give us a way to define  the Bowditch set $\X_Q(\Sigma_{g,b})$ as follows:
$$\X_Q(\Sigma_{g,b}) = \{\rho \in \X(F_{2g+b-1}) \mid \exists k = k(\rho), \; |\mathrm{L}(\rho(\g))| \geq k \mathrm{W}(\g)\;\;\forall \g \in \Sc(\Sigma_{g,b})\},$$
where $\Sc(\Sigma_{g,b})$ is the set of (free homotopy classes of) non-trivial, non-peripheral simple closed curves in $\Sigma_{g,b}$. This definition could also help in understanding the relationship with the set of primitive-stable elements, as we discussed in Proposition \ref{pr:primitivestable}, and likewise, this set can be defined in different ways, see Section \ref{s:characterization}. 

The problem that arises for this generalization is that the proof of the proper discontinuous action of the mapping class group in the simple cases comes from the simple combinatorial description of the complex of curves in the cases that are considered. Hence it will be interesting to understand the right combinatorial object (replacing the graph $\Upsilon$, dual to the complex of curves of $N$), because that will allow a generalization of most of the results included here. Note that in the cases analyzed, the surfaces had all small complexity, and so (2-sided) simple closed curves are in correspondence with pants decompositions, and a vertex of $\Upsilon$ (or $\Sigma$ in \cite{bow_mar}, \cite{tan_gen} and \cite{mal_ont}) corresponds to a triangulation of the surface. In particular, simple closed curves are useful in trace relations, while triangulations are useful to define `flips', which are related to our edge relations, so we need a combinatorial view point which keep both these two points of view.

\subsection{Real characters}

It is interesting to focus on real characters, that is, representations of $\pi_1 (S)$ into one of the two real forms of $\SLtwoC$, namely $\SLtwoR$ and $\mathrm{SU}(2)$. Previous work of Goldman \cite{gol_erg} and the second author \cite{pal_erg} prove that the mapping class group action is ergodic on the relative $\mathrm{SU}(2)$ character varieties for all compact hyperbolic surfaces, orientable or not, with the exception of $N_{3,0}$, $N_{1,2}$ and $N_{2,1}$. The $\SLtwoR$ case is much more complex, because one can expect to have domains of discontinuity as well as domains where the action is ergodic. In fact, a complete description of the dynamical decomposition of the action is still unknown in general.

The case of the free group $F_2$ of rank two has been studied by Goldman \cite{gol_the} and Goldman, McShane, Tan and Stantchev \cite{gol_dyn}, who gave a complete description of the dynamics of the action of $\Out( F_2)$ on the real characters. In \cite{mal_ont}, we proved few results about the real case $\X^{\R}(\Sigma_{0,4})$, but it would be interesting to give a complete dynamical decomposition for the mapping class group action on $\X^{\R}(F_3)$, like in \cite{gol_the}. In addition, it would be also interesting to generalize the work of \cite{gol_dyn} by considering all five surfaces which have $F_3 $ as their fundamental group. (Note that there are two orientable surfaces, and three non-orientable surfaces with fundamental group $F_3$.)

In the case of the free group of rank two, the results of \cite{gol_the} prove that the real representations in the domain of discontinuity for the action of $\Out (F_2)$ all come from geometric structures on a surface with fundamental group $F_2$, including hyperbolic structures with conical points. So a very interesting question would be to understand what happens for $F_3$.   

\subsection{Torelli group action}
Another follow-up project that we are planning to study is the action of the Torelli group $\Tc_n$ on $\X$. As explained in the Introduction, and in Remark \ref{Torelli}, the Torelli group $\Tc_3$ is generated by seven involutions whose actions on character is easily described. Hence, the two actions studied here and in \cite{mal_ont} can be combined to study the action of $\Tc_n$ on $\X$. A natural idea is to consider the intersection of both Bowditch sets, which are domains of discontinuity for the $\MCG(N_{1,3})$-action and the $\MCG(S_{0,4})$-action on $\X (F_3)$.

\subsection{McShane identities}
Finally, as pointed out in \cite{mal_ont}, it should be possible to describe some new McShane's identities for the four-holed sphere and for the three-holed projective plane. In \cite{hua_sim}, Huang and Norbury described an interesting identity for the case of three punctured projective plane. However, the punctured case is easier since the associated equation is very symmetric. So it would be interesting to start looking at surfaces where the boundary components have all the same trace.


\appendix


\renewcommand{\thesection}{\Alph{section}}

\section{Explicit expression of the $H$ function}\label{app:functionH}

In this section we give an explicit expression for the function $H$ defined in section \ref{sec:att_tree}. We state it in a more general context, where it could be applied to the functions $H$ defined in Tan-Wong-Zhang \cite{tan_gen} and Maloni-Palesi-Tan \cite{mal_ont}. Note that all functions are defined on the set of faces and take values in $\R$. 

Recall that the setting is the one that is found in Section \ref{s:face}, where we consider the two bi-infinite sequence of neighboring regions around a given face $\xi$. The edge relations on consecutive edges gives a recurrence relation for these bi-infinite sequence. As this setting is similar to the situation for the one-holed torus and the four-holed sphere, we state the expression of the $H$ function for sequences satisfying a particular recurrence relation. 

Let $(Q, R, S, X) \in \C^4$ be parameters, and consider the sequences $(y_n)$ and $(z_n)$ defined by the recurrence relation:
$$\begin{pmatrix} y_{n+1} \\ z_{n+1} \end{pmatrix} =\begin{pmatrix} -1 & -X \\ X &  X^2 - 1  \end{pmatrix}   \begin{pmatrix} y_{n} \\ z_{n} \end{pmatrix} + \begin{pmatrix} Q \\ R-XQ \end{pmatrix},$$
with $y_0$ and $z_0$ satisfying the equation:
$$y_0^2+z_0^2+Xy_0z_0-Qy_0-Rz_0 = S.$$

The Lemma \ref{lem:neighbors} implies that there exists a constant $H =  H (Q,R,S, X) \in \R_{>0}\cup \{ \infty \}$, depending only on the values $(Q,R,S,X)$, such that there are integers $n_1 \leq n_2$ satisfying $|y_n| \leq H$ if and only if $n_1 \leq n \leq n_2$ and $|y_n|$ is monotonically increasing for $n \geq n_2$ and monotonically decreasing for $n \leq n_1$.

This constant is finite for $X \notin [ -2 , 2 ]$, and $ Q^2+R^2-XRQ+S(X^2-4) \neq 0$.

\begin{Lemma} $H$ can be chosen as:
$$H(Q,R,S, X) = \sqrt{|T|} |\lambda| (W +1) + |\eta|,$$
where $\lambda$ satisfies $\lambda + \lambda^{-1}= X^2 - 2$, and the other terms are defined by:
\begin{align*}
	T & = \dfrac{Q^2+R^2-XRQ+S(X^2-4)}{(X^2-4)^2}\\
	\eta & = \dfrac{2Q - XR}{X^2-4}\\
	W & =  \dfrac{1}{\sqrt{|T|} |\lambda| (|\lambda| - 1) } \left(  |\eta| + \sqrt{|\eta|^2 - |\lambda|(|\lambda|^2 - 1 )} \right). 
\end{align*}
\end{Lemma}

\begin{proof}
From the recurrence relation satisfied by $y_n$ and $z_n$ we can deduce that:
$$y_n = A \l^n + B \l^{-n} + \eta,$$
with $AB = T$.

We can assume, without loss of generality, that 
$$\sqrt{|T|} | \l | ^{-1} \leq |A| , |B| \leq \sqrt{|T|} | \l |,$$
(up to reparametrization in $n$). 

Suppose that  $|\l |^n > W$. Then we have that 
$$|\l |^{2n} - \frac{2\eta}{\sqrt{T} |\l| ( |\l | - 1 )} |\l |^n - |\l | \frac{|\l|+1}{|\l |-1} \geq 0.$$
We deduce that 
$$|\l |^{2n+1} \geq \frac{|B| |\l|+1}{|A| |\l |-1} + \frac{2\eta}{|A| ( |\l | - 1 )} |\l |^{n+1}.$$
This can be interpreted as 
$$|y_{n+1}| \geq |A| |\l |^{n+1} - |B| |\l |^{-n-1} - |\eta | \geq |A| |\l |^{n} + |B| |\l |^{n} + |\eta | \geq  |y_n|.$$

On the other hand we have:
\begin{align*}
	|y_n| & \leq |A \l^n + B \l^{-n} | + |\eta| \\
		& \leq \sqrt{|AB|} |\l | ( |\l |^n + |\l |^{-n} ) + |\eta |\\
		& \leq \sqrt{|AB|} |\l | ( |\l |^n + 1 ) + |\eta |.
\end{align*}
In conclusion,  if $| y_n | \geq \sqrt{|AB|} |\l | ( W + 1 ) + |\eta | = H$, then we have $|\l |^n \geq W$, and we can apply the previous discussion to show that $|y_{n+1}| \geq |y_n|$. 

In a similar way, we can prove that, if $|y_{-n}| \geq W$, then $|y_{-n-1} | \geq  | y_{-n} |$.
\end{proof}

\begin{Remark} We can also define a function $H'$ for the sequence $z_n$ by permuting the constant $Q$ and $R$. In other words $H' (Q,R,S,X) = H(R,Q,S,X)$.
\end{Remark}

Now we can apply this formula in several cases. We will follow the notation used in the papers where the respective case is discussed.  

\begin{enumerate}
	\item	\textit{The one-holed torus $S_{1,1}$} (Bowditch \cite{bow_mar}, Tan-Wong-Zhang \cite{tan_gen}):\\
	We consider a one-holed torus $S_{1,1}$, with boundary trace equal to $\mu$, that is, we consider the relative character variety $\X_{(a,b,c,d)}(S_{0,4})$. Let's fix a face $X$ such that $\phi(X) = x$. We look at the neighbors of $X$. In this case, we can set: 
	\begin{align*}
		Q &= R = 0\\\
		X &=x\\
		S &= \mu-x^2, 
	\end{align*}
and the formula becomes	
$$H(0, 0, \mu - x^2 , x) = \sqrt{\dfrac{| x^2 - \mu |}{|x^2 - 4|} } \dfrac{2 |\l|^2}{|\l |-1} $$
This corresponds exactly to the function used by Tan-Wong-Zhang (\cite{tan_gen}, Lemma 3.20). 

	\item \textit{The four-holed sphere $S_{0,4}$} (Maloni-Palesi-Tan \cite{mal_ont}):\\
	We consider a four-holed sphere $S_{0,4}$ with boundary traces $(a,b,c,d)$, that is, we consider the relative characyter variety $\X_{(a,b,c,d)}(S_{0,4})$.  Instead of $\Upsilon$, the graph is called $\Sigma$, and is a trivalent simplicial tree in $\mathbb{H}^2$, and the set of faces (or of connected components of $\mathbb{H}^2\setminus\Sigma$) is denoted $\Omega(\Sigma)$. Let's fix a face $X \in \Omega_1(\Sigma)$ (of colour $1$) such that $\phi(X) = x$. We look at the neighbors of $X$. The expression of $(Q,R,S, X)$ for regions in $\Omega_2(\Sigma)$ or $\Omega_3(\Sigma)$ can be found with a similar method.
	In this case, we can use:
	\begin{align*}
		Q &= bc+ad\\
		R &= ac+bd\\
		X &= x\\
		S &= 4-a^2-b^2-c^2-d^2-abcd-(ab+cd)x-x^2. 
	\end{align*}

	\item \textit{The three-holed projective plane $N_{1,3}$} (Maloni--Palesi):\\
We consider a three-holed projective plane $N_{1,3}$ with boundary traces $(x,y,z)$, that is, we consider the relative character variety $\X_{(x,y,z)}(N_{1,3})$. Let's fix a face $\xi = \xi_{(\a,\b)} \in \Upsilon^{(2)}$ such that $\xi = \a \cap \b$ with $\a\in\Upsilon^{(3)}_1$ and $\b\in\Upsilon^{(3)}_2$ such that $\psi(\a) = a$ and $\psi(\b) = b$. We look at the neighbors of $\xi$. The expression of $(Q,R,S, X)$ for regions in $\Omega_2(\Sigma)$ or $\Omega_3(\Sigma)$ can be found with a similar method.
	In this case, we can use:
	\begin{align*}
		Q &= yb+az\\
		R &= ya+zb\\
		X &= ab-x\\
		S &= 4-a^2-b^2-x^2-y^2-z^2-xyz-xab. 
	\end{align*}
\end{enumerate}

The formula gives an explicit method to determine if an edge of the 1-skeleton is in the attracting subtree $T_\phi (2+M)$. As the tree is connected, one can produce an explicit algorithm to determine if a Markoff map is in the Bowditch set. This is one of the necessary tool if one wants to generate computer pictures of the Bowditch set.


\section{Relation with the Torelli subgroup}\label{app:Torelli}

\subsection{Representatives of the involutions}


In the Introduction, we defined seven involutions on the character variety of the free group $F_3$ of rank three, which are given by seven involutions in $\Out (F_3)$. We can  give explicit representatives of the involutions in $\mathrm{Aut} (F_3)$. 


\begin{align*}
	& \widetilde{\tau}_a : \begin{pmatrix} A \\ B \\ C \end{pmatrix} \mapsto \begin{pmatrix} C B^{-1} A^{-1} C^{-1} B \\ B^{-1}  \\ C^{-1} \end{pmatrix}, & & 
		 \widetilde{\tau}_x : \begin{pmatrix} A \\ B \\ C \end{pmatrix} \mapsto \begin{pmatrix} A^{-1} \\ C^{-1} B^{-1} C \\ C^{-1} \end{pmatrix}, & & \\
	& \widetilde{\tau}_b : \begin{pmatrix} A \\ B \\ C \end{pmatrix} \mapsto \begin{pmatrix} A^{-1} \\ A C^{-1} B^{-1} A^{-1} C  \\ C^{-1} \end{pmatrix}, & &
		 \widetilde{\tau}_y : \begin{pmatrix} A \\ B \\ C \end{pmatrix} \mapsto \begin{pmatrix} A^{-1} \\  B^{-1}  \\ A^{-1} C^{-1} A  \end{pmatrix},& &
		 \widetilde{\tau}_d : \begin{pmatrix} A \\ B \\ C \end{pmatrix} \mapsto \begin{pmatrix} A^{-1} \\ B^{-1}  \\ C^{-1} \end{pmatrix}, \\
	& \widetilde{\tau}_c : \begin{pmatrix} A \\ B \\ C \end{pmatrix} \mapsto \begin{pmatrix} A^{-1} \\ B^{-1}  \\ B A^{-1} C^{-1} B^{-1} A \end{pmatrix}, & &
		\widetilde{\tau}_z : \begin{pmatrix} A \\ B \\ C \end{pmatrix} \mapsto \begin{pmatrix} B^{-1} A^{-1} B \\  B^{-1}  \\ C^{-1} \end{pmatrix}. & & \\
\end{align*}

These are representatives of maps $\tau_a, \dots , \tau_z$ in $\Out (F_3)$. The action of $\tau_a, \dots , \tau_z$ on the character variety $X (F_3)$ is exactly given by the seven involutions $\theta_a , \dots , \theta_z$ defined in the introduction.

\subsection{Torelli subgroup}

There is a natural surjective homomorphism from $\Out (F_n)$ to $\mathrm{GL}(n, \Z)$, given by abelianizing the free group $F_n$. Its kernel is known as the {\em Torelli subgroup} of $\Out (F_n)$, and we denote it by $\Tc_n$:
$$1 \rightarrow \Tc_n \rightarrow \Out(F_n) \rightarrow \mathrm{GL}(n, \Z) \rightarrow 1 $$
We can also define the subgroup $\Tc_n'$ as the inverse image of the center $\{ \pm I_n \}$ by the abelianization map: 
$$1 \rightarrow \Tc_n' \rightarrow \Out(F_n) \rightarrow \mathrm{PGL}(n, \Z) \rightarrow 1 $$

It is clear that $\Tc_n$ is an index two subgroup of $\Tc_n'$. This group is generated by $\Tc_n$ together with the involution $(A_1 , \dots , A_n) \mapsto (A_1^{-1} , \dots, A_n^{-1} )$.

\begin{Proposition}
The group $\Gamma$ generated by the seven involutions $\tau_a, \dots, \tau_z$ is equal to $\Tc_3'$
\end{Proposition}

\begin{proof}
	First, it is clear that each involution $\tau_a, \dots, \tau_z$ is an element of $ \Tc_n'$. It remains to show that the generators of $\Tc_n'$ can be written as a product of the involutions. 
	We use the Magnus generating set of  $\Tc_n$ (see Bestvina-Bux-Margalit \cite{bes_dim}) given by their lift in $\mathrm{Aut} (F_n)$:
	$$ K_{ij}  : \begin{pmatrix}  A_1 \\  \vdots \\ A_i \\ \vdots \\ A_n \end{pmatrix}  \mapsto \begin{pmatrix}  A_1 \\  \vdots \\ A_j A_i A_j^{-1} \\ \vdots \\ A_n \end{pmatrix}
	\hspace{1cm}
		K_{ijk} : = \begin{pmatrix}  A_1 \\  \vdots \\ A_i \\ \vdots \\ A_n \end{pmatrix} \mapsto \begin{pmatrix}  A_1 \\  \vdots \\ A_i [ A_j , A_k ] \\ \vdots \\ A_n \end{pmatrix} $$
	
	We have $K_{ijk} =K_{ikj}^{-1}$ in $\Out (F_n)$, and, when $n = 3$, we have $K_{ij} = (K_{kj})^{-1}$ in $\Out (F_3)$. So we have the six generators for $\Tc_3$ given by:
	$$K_{12}, K_{23} , K_{31}, K_{123}, K_{231}, K_{312}.$$
	
	These six generators together with $\tau_d$ generate $\Tc_3'$. These generators can now be obtained as products of the involutions as follows:
	\begin{align*}
  	K_{12} &= \tau_z \circ \tau_d, \\
  	K_{23} &= \tau_x \circ \tau_d,\\
  	K_{31} &= \tau_y \circ \tau_d, \\
  	K_{123} &= \tau_d \circ \tau_x \circ \tau_a \circ \tau_z,\\
  	K_{231} &= \tau_d \circ \tau_y \circ \tau_b \circ \tau_x,\\
  	K_{312} &= \tau_d \circ \tau_z \circ \tau_c \circ \tau_y.
  \end{align*}
	This proves that the two groups $\Gamma$ and $\Tc_3'$ are equal.
\end{proof}

\bibliographystyle{plain}

\bibliography{sample2}

\end{document}